\documentclass[draft]{amsart}

\usepackage{amssymb}
\usepackage{url}
\usepackage{amscd}

\usepackage{amssymb} 
\usepackage{amsthm}
\usepackage{amsmath}

\usepackage{latexsym}
\usepackage{mathrsfs}
\usepackage{enumerate}
\usepackage[dvips]{graphicx}
\usepackage{color}

\definecolor{mgre}{cmyk}{0.92,0.00,0.59,0.25}

\def\bH{{\mathbb{H}}}

\def\bR{{\mathbb{R}}}
\def\bS{{\mathbb{S}}}

\def\bfE{{\mathbf{E}}}

\def\bfP{{\mathbf{P}}}

\def\bfR{{\mathbf{R}}}

\def\cA{{\mathcal{A}}}

\def\cC{{\mathcal{C}}}

\def\cE{{\mathcal{E}}}
\def\cF{{\mathcal{F}}}

\def\cK{{\mathcal{K}}}
\def\cL{{\mathcal{L}}}
\def\cM{{\mathcal{M}}}

\def\cO{{\mathcal{O}}}

\def\cS{{\mathcal{S}}}

\def\sB{{\mathscr{B}}}

\def\sF{{\mathscr{F}}}

\def\1{{\mathbf{1}}}

\newcommand{\dis}{\displaystyle}

\newcommand{\loc}{\textrm{loc}}

\def\<{\langle}
\def\>{\rangle}
\newtheorem{thm}{Theorem}[section]

\newtheorem{prop}[thm]{Proposition}
\newtheorem{lem}[thm]{Lemma}
\newtheorem{cor}[thm]{Corollary}

\theoremstyle{definition}
\newtheorem{defi}[thm]{Definition}
\newtheorem{prob}{}[section]
\newtheorem{rem}[thm]{Remark}
\newtheorem{ex}[thm]{Example}
\newtheorem{pro}{Problem}[section]
\newtheorem{ass}[thm]{Assumption}

\newcommand{\bd}{\begin{defi}}
\newcommand{\ed}{\end{defi}}

\newcommand{\bpro}{\begin{pro}}
\newcommand{\epro}{\end{pro}}

\newcommand{\bec}{\begin{cases}}
\newcommand{\eec}{\end{cases}}

\newcommand{\bpr}{\begin{prob}}
\newcommand{\epr}{\end{prob}}

\newcommand{\bt}{\begin{thm}}
\newcommand{\et}{\end{thm}}

\newcommand{\ba}{\begin{ass}}
\newcommand{\ea}{\end{ass}}

\newcommand{\br}{\begin{rem}}
\newcommand{\er}{\end{rem}}

\newcommand{\bpm}{\begin{pmatrix}}
\newcommand{\epm}{\end{pmatrix}}

\newcommand{\be}{\begin{ex}}
\newcommand{\ee}{\end{ex}}

\newcommand{\bp}{\begin{prop}}
\newcommand{\ep}{\end{prop}}

\newcommand{\bl}{\begin{lem}}
\newcommand{\el}{\end{lem}}

\newcommand{\bc}{\begin{cor}}
\newcommand{\ec}{\end{cor}}

\newcommand{\bq}{\begin{que}}
\newcommand{\eq}{\end{que}}

\newcommand{\beqn}{\begin{eqnarray*}}
\newcommand{\eeqn}{\end{eqnarray*}}

\newcommand{\beqnn}{\begin{eqnarray}}
\newcommand{\eeqnn}{\end{eqnarray}}

\newcommand{\bequ}{\begin{equation}}
\newcommand{\eequ}{\end{equation}}

\newcommand{\benu}{\begin{enumerate}}
\newcommand{\eenu}{\end{enumerate}}

\newcommand{\barr}{\begin{array}{rcl}}
\newcommand{\ear}{\end{array}}

\newcommand{\la}{\label}

\newcommand{\eps}{\epsilon}

\newcommand{\Om}{\Omega}

\usepackage{amsmath}	

 \makeatletter
    
    \@addtoreset{equation}{section}
  \makeatother
\begin{document}

\newcommand\sfrac[2]{{#1/#2}}

\newcommand\cont{\operatorname{cont}}
\newcommand\diff{\operatorname{diff}}

	
	\title [Maximum Principle]{Explosion by Killing and Maximum Principle in Symmetric Markov Processes}
	
	\author{Masayoshi Takeda}
	\address{Department of Mathematics,
		Kansai University, Yamatecho, Suita, 564-8680, Japan}
	\email{mtakeda@kansai-u.ac.jp}
	\thanks{The author was supported in part by Grant-in-Aid for Scientific
		Research (No.18H01121(B)), Japan Society for the Promotion of Science.}
	
	
	
	\begin{abstract}
Keller and Lenz \cite{KL} define a concept of {\it stochastic completeness at infinity} (SCI) for  
a regular symmetric Dirichlet form $(\cE,\cF)$. We show that (SCI) can be characterized probabilistically  
by using the predictable part $\zeta^p$ of
the life time $\zeta$ of the symmetric Markov process $X=({\bf P}_x,X_t)$ 
generated by 
$(\cE,\cF)$, that is, (SCI) is equivalent to $\bfP_x(\zeta=\zeta^p<\infty)=0$. We define a concept, 
 {\it explosion by killing} (EK), by $\bfP_x(\zeta=\zeta^i<\infty)=1$. Here $\zeta^i$ is the 
  totally inaccessible part of $\zeta$. We see that (EK) is equivalent to (SCI) and $\bfP_x(\zeta=\infty)=1$.  
 Let $X^{\rm res}$ be the {\it resurrected process} generated by the  {\it resurrected form}, a regular Dirichlet form 
 constructed by removing the killing part from $(\cE, \cF)$.   
  Extending work of Masamune and Schmidt (\cite{MS}), we show that (EK) is also 
 equivalent to the ordinary conservation property of time changed process of  $X^{\rm res}$ by $A^k_t$,
 where the $A^k_t$ is the positive continuous additive functional in the Revuz correspondence 
 to the killing measure $k$ in the Beurling-Deny formula (Theorem \ref{ma-sh}).

 We consider the maximum principle for Schr\"odinger-type operator $\cL^\mu=\cL-\mu$.  
 Here $\cL$ is the self-adjoint operator 
 associated with  $(\cE,\cF)$ 
 and $\mu$ is a  
 Green-tight Kato measure. Let $\lambda(\mu)$ be the principal eigenvalue 
  of the trace of $(\cE,\cF)$ 
 relative to $\mu$. We prove that if (EK) holds, then $\lambda(\mu)>1$ implies a Liouville property 
  that every bounded solution to $\cL^\mu u=0$
  is zero quasi-everywhere and that 
  the {\it refined maximum principle} in the sense of Berestycki-Nirenberg-Varadhan
 \cite{BNV} holds for $\cL^\mu$ if and only if $\lambda(\mu)>1$ (Theorem \ref{RMP}).
 .
 
		\end{abstract}
	
	\maketitle
	
	\section{Introduction}
In \cite{T-B}, \cite{T4}, we prove the maximum principle for Schr\"odinger forms,
 and in Kim and Kuwae \cite{KK}, they extend our results to more general 
  Schr\"odinger forms with distributions as potential.
   Our aim in this paper is to extend results in \cite{T-B}, \cite{T4} to 
 more general class of subsolutions. 		
	
Let $E$ be a locally compact separable metric space and $m$ 
a positive Radon measure on $E$ with full topological support. 
Let $(\cE, \cF)$ be a transient, regular and symmetric Dirichlet form 
on $L^2(E;m)$ and $X = (\bfP_x, X_t, \zeta)$ the 
$m$-symmetric Hunt process generated by $(\cE, \cF)$. Here $\zeta$ is 
the life time of $X$, $\zeta=\inf\{t>0\mid X_t\not\in E\}$. 
We, in addition, assume that 
$X$ is irreducible and strong Feller.
We take a point $\Delta$ not in $E$ as a cemetery. Any function $u$ 
on $E$ is always extended to a function on $E_\Delta(=E\cup \{\Delta\})
$ by setting $u(\Delta)=0$. When $E$ is not compact, we
write $E_\infty$ for the one point comactification of $E$.  

Let $\mu$ be a non-trivial measure in the Green-tight Kato class $\cK_\infty$
with respect to $X$ (For the definition of 
  $\cK_\infty$, see Definition 
  \ref{def-Kato} below.)
For $\mu\in\cK_\infty$,
 we define a
Schr\"{o}dinger form  $(\cE^\mu , \cF^\mu(=\cF))$ on $L^2(E;m)$ by 
\begin{equation*}
\cE^\mu(u,v)=\cE(u,v)-\int_E\widetilde{u}\widetilde{v}d\mu,\ \ u,v\in\cF^\mu.
\end{equation*}
Here $\widetilde{u}$ represents a quasi-continuous version of $u\in\cF$. 
We denote by $\cF_{\loc}$ the set of functions locally in $\cF$. Each function $u$ in 
$\cF_{\loc}$ admits a quasi-continuous version $\widetilde{u}$. In the sequel, 
we suppose that $u\in\cF_{\loc}$ is already modified and write $u$ for $\widetilde{u}$ simply. 

\medskip
Define a measure on $E$ by
\bequ\la{mj}
\mu^j_{\langle u\rangle}(B)=\iint_{B\times E}(u(x)-u(y))^2J(dx,dy),\ \ B\in\sB(E),
\eequ 
where $J$ is the jumping measure in the Beurling-Deny formula for $(\cE,\cF)$ (\cite[Theorem 3.2.1]{FOT})
and $\sB(E)$ is the set of Borel subsets of $E$. 
 Following \cite{FL}, \cite{Ku}, we introduce a subspace $\cF^\dagger_{\loc}$ 
 of $\cF_{\loc}$:
$$
\cF^\dagger_{\loc}=\left\{u\in\cF_{\loc}\,\Big|\, \mu^j_{\langle u\rangle}\ 
\text{is a Radon measure on $E$}\ \right\}.
$$
We then see that a 
bounded function in ${{\cF}}_{\text{loc}}$ belongs to ${{\cF}}^\dagger_{\text{loc}}$ and that  
 for $u\in{\cF}^\dagger_{\loc}$ and $\varphi \in \cF\cap C_0(E)$, 
$\cE(u,\varphi )$ is well-defined (\cite[Theorem 3.5]{FL}). Here $C_0(E)$ is the set of continuous 
functions on $E$ with compact support.

Define $\cL^{\mu}$ as the self-adjoint 
operator associated with the closed symmetric form 
$(\cE^\mu,\cF^\mu)$, $(-\cL^{\mu}u,v)_m=\cE^\mu(u,v)$. 
 A function $h\in\cF^\dagger_{\loc}\cap L^\infty_{\loc}$ is 
called a {\it solution} ({\it subsolution}, {\it supersolution}) to
$\cL^\mu u=0$ if 
$$
\cE^\mu(h,\varphi)=0 \ (\leq 0,\ \geq 0)\ \ \text{for all}\ \varphi\in\cF_+\cap C_0(E).
$$
Here $L^\infty_{\loc}$ is the set of locally bounded $m$-measurable functions on $E$ and 
for a function space $\cA$, $\cA_+\,(\cA_-)$ represents the totality 
of non-negative (non-positive) functions in $\cA$. 
We write ${\bf S}^{\mu,\rm sub}\,({\bf S}^{\mu,\rm sup})$ for 
the space of {subsolutions} ({supersolutions}) 
and define the function spaces:
\begin{align}\la{sol-0}
\cS^{\mu,\rm sub}&= \{h\in {\bf S}^{\mu,\rm sub}\mid   \|h^+\|_\infty<\infty\}.\\
\cS^{\mu,\rm sup}&= \{h\in {\bf S}^{\mu,\rm sup}\mid   \|h^-\|_\infty>-\infty\}.\\
\cS&=\cS^{\mu,\rm sub}\cap\cS^{\mu,\rm sup}
\end{align}
From now on, we will mainly discuss the space $\cS^{\mu,\rm sub}$ because $-h\in\cS^{\mu,\rm sub}$
for $h\in\cS^{\mu,\rm sup}$.

\bigskip
 We set 
\begin{align}
{\bS}&=\left\{\{x_n\}_{n=1}^\infty \subset E\mid
\text{
$\lim_{n\to\infty }\bfE_{x_n}\left(e^{-\zeta}\right)=1$}
\right\}
\end{align}
and introduce a subspace of $\cS^{\mu,\rm sub}$ by
\begin{equation*}
{\cS}_0^{\mu,\rm sub}=\left\{h\in{\cS}^{\mu,\rm sub}\cap C(E)\mid \varlimsup_{n\to\infty }h(x_n)\leq 0
\ \text{for all}\ \{x_n\}_{n=1}^\infty \in \bS\right\}.
\end{equation*}

\medskip
Following  Berestycki, Nirenberg and Varadhan \cite{BNV}, we define the {\it refined maximum principle}:

\bigskip
\noindent
\ \ \ \ \ (${\bf{RMP}}$)\ \ If $h\in {\cS}_0^{\mu,\rm sub}$, then $h(x)\leq 0$ for all $x\in E$.

\bigskip

\smallskip
\noindent

 We then obtain the following main theorem:
\begin{thm}\la{RMP1} For $\mu\in\cK_\infty$ 
	define 
	\begin{equation}\la{j-ground2}
		\lambda (\mu)= \inf \left\{ \cE(u)\, \Big|\  u \in
		\cF,\ \int_E u^2 d\mu = 1\right\},
	\end{equation}
	where $\cE(u)=\cE(u,u)$. Then 
$$
\lambda (\mu)>1 \Longleftrightarrow ({\bf RMP}).
$$
\end{thm}
$\lambda (\mu)$ is identified with the principal eigenvalue of the trace of $(\cE,\cF)$ relative to $\mu$,
in other words, the principal eigenvalue of the time changed process,
$X_{\tau_t},\ \tau_t=\inf\{s>0\mid A^\mu_s>t\}$,  
where $A^\mu_t$ is the positive continuous additive functional (PCAF in abbreviation) of $X$  
in the Revuz correspondence to the measure $\mu$ (cf. \cite[Section 6.2]{FOT}). 
In the sequel, for a symmetric form $a(u,v)$ we simply write 
$a(u)$ for $a(u,u)$. 

The differences between Theorem \ref{RMP1} and the corresponding theorems in 
\cite{T-B}, \cite{T4} are as follows: In this paper we will deal with general Dirichlet forms 
 with non-local part, while in \cite[Theorem 3.2]{T-B}, \cite[Remark 4.3]{T4} Dirichlet forms are supposed to be 
 strongly local. Moreover, when general Dirichlet forms with non-local 
 part are dealt with, each $h\in {\cS}_0^{\mu,\rm sub}$
  is supposed to be in $\cF_e\cap C(E)(\subset \cF^\dagger_{\loc}\cap C(E)$ (\cite[Theorem 4.1]{T4}). 
  Here, $\cF_e$ is the {\it 
  extended Dirichlet space} of $(\cE,\cF)$.
  
Let $\gamma(\mu)$ be the principal eigenvalue of the Schr\"odinger operator $-\cL^\mu$:
\begin{equation}\la{j-ground3}
	\gamma (\mu)= \inf \left\{ \cE^\mu(u)\, \Big|\  u \in
	\cF^\mu,\ \int_E u^2 dm = 1\right\}.
\end{equation}
If, in addition, 
 the basic measure 
$m$ belongs to $\cK_\infty$, then it holds that $\lambda(\mu)>1$ if and only if 
$\gamma(\mu)>0$ (Theorem \ref{r-g}). 
It is shown in \cite{BNV} that for an elliptic (not necessarily symmetric) operator 
in a bounded domain of $\bR^d$,  ({\bf RMP}) holds if and only 
if the generalized principal eigenvalue, an extension of $\gamma(\mu)$ defined 
through the so-called {\it Donsker-Varadhan {I}-function}, is positive.

We here note that the semi-group of the self-adjoint operator $\cL^\mu$, $T^\mu_t:=\exp(t\cL^{\mu})$, 
is expressed 
by the Feynman-Kac semi-group: For $f\in{\mathscr B}_b(E)\cap L^2(E;m)$
$$
T^\mu_tf(x)=p^\mu_tf(x):=\bfE_x\left(e^{A^{\mu}_t}f(X_t)\right) \ \ m\text{-a.e.}\ x,
$$
where ${\mathscr B}_b(E)$ is the space of bounded Borel functions on $E$.
Then the key to the proof of Theorem \ref{RMP1} is as follows: For showing the direction 
$(\Longrightarrow)$, the first crucial fact is that
 if $h$ belongs to $\cS_0^{\mu,\rm sub}$,
 then 
 \begin{align} \la{ku}
h(x)&\le\bfE_x\left(e^{A^\mu_{\sigma_n}}h^+(X_{\sigma_n})\right)\ \ {\rm q.e.}\ x.
\end{align}
 Here $h^+=h\vee 0$ and $\{\sigma_n\}$ is a certain sequence of stopping times such that 
 $\sigma_n\uparrow \zeta$. 
 ``q.e.'' is an abbreviation of ``{\it quasi-everywhere}''(cf. \cite[Section 2.1]{FOT}).
The proof of (\ref{ku}) is given by an application of the Kuwae's 
 {generalized Fukushima decomposition} to the {\it resurrected Hunt process} 
$X^{\rm res}= (\bfP^{\rm res}_x, X_t, \zeta)$ generated by the  {\it resurrected form}  
$(\cE^{\rm res}, \cF^{\rm res})$. Here the {resurrected form}  
$(\cE^{\rm res}, \cF^{\rm res})$ is a regular Dirichlet form 
constructed by removing the killing part from $(\cE, \cF)$ (\cite[Theorem 5.2.17]{CF}).
The life time $\zeta$ of $X$ is written as $\zeta=\zeta^p\wedge\zeta^i$ 
on $\{\zeta<\infty\}$, where $\zeta^p$ and 
$\zeta^{i}$ are the {\it predictable} part and the {\it totally inaccessible} part of $\zeta$. It holds true that  
 $X_{\zeta-}=\infty$ on $\{\zeta=\zeta^p<\infty\}$ and $X_{\zeta-}\in E$ on $\{\zeta=\zeta^i<\infty\}$.
The killing measure $k$ in the Beurling-Deny formula of $(\cE, \cF)$ is identified with 
the Revuz measure corresponding to the PCAF $\left(1_{\{X_{\zeta-}\neq\infty,\,\zeta\leq t\}}\right)^{p}$, the dual predictable 
projection of $1_{\{X_{\zeta-}\neq\infty,\,\zeta\leq t\}}$ (cf. \cite[Theorem 5.3.1]{FOT}). 
 Hence the life time of $X^{\rm res}$ 
turns out to be a {predictable} stopping time because
 $(\cE^{\rm res}, \cF^{\rm res})$ has no killing part. In other words, 
the removal of the killing part from the Dirichlet form $(\cE,\cF)$ is 
equivalent to that of the {totally inaccessible} part from the life time $\zeta$. 
The original process $X$ is 
 the subprocess of $X^{\rm res}$ generated by  
the multiplicative functional $\exp(-A^k_t)$, where $A^k_t$ is the PCAF of $X^{\rm res}$  
 corresponding to the killing measure $k$ (\cite[Theorem 5.2.17]{CF}).
 These facts enable us to imitate the argument in the case of strongly local Dirichlet forms.

The second crucial fact for the direction 
$(\Longrightarrow)$ is that the condition $\lambda(\mu)>1$ is equivalent with the gaugeability (\cite{C}):
 \bequ\la{gau}
 \lambda(\mu)>1 \Longleftrightarrow \sup_{x\in E}\bfE\left(e^{A^\mu_\zeta}\right)<\infty.
 \eequ
 Owing to the gaugeability, the reverse Fatou lemma is applicable to the equation
 (\ref{ku}) and
\bequ\la{Fa}
h(x)\le\varlimsup_{n\to\infty}\bfE_x\left(e^{A^\mu_{\sigma_n}}h^+(X_{\sigma_n})\right)
\le\bfE_x\left(e^{A^\mu_{\zeta}}\varlimsup_{n\to\infty}h^+(X_{\sigma_n})\right)\ \ {\rm q.e.}\ x.
\eequ
We will show in Lemma \ref{s-i} that for $x\in E$
$$
\{X_{\sigma_n}\}_{n=1}^\infty\in \bS\ \ \bfP_x\text{-a.s.}\ \text{on}\ \ 
\displaystyle{\cap_{n=1}^\infty}\{\sigma_n<\zeta\}. 
$$
Hence noting that 
$$
\varlimsup_{n\to\infty}h(X_{\sigma_n})\le 0\ \Longleftrightarrow\ \varlimsup_{n\to\infty}h^+(X_{\sigma_n})=0,
$$
we see the right hand side of (\ref{Fa}) equals 0. 

For the proof of the opposite direction 
$(\Longleftarrow)$, we apply the fact that for $\mu\in\cK_\infty$ the extended Dirichlet space $\cF_e$ is 
 compactly embedded in $L^2(E;\mu)$, which leads us to that 
 there exists a strictly positive function 
$h\in\cF_e\cap C(E)$  
 attaining the infimum of (\ref{j-ground2}). Moreover, if $\lambda(\mu)\le 1$, then the function $h$ 
 satisfies $p^\mu_th\ge h$ (Lemma \ref{e-2}) and for any $\{x_n\}\in{\bS}$, 
 $\varlimsup_{n\to\infty }h(x_n)\leq 0$ (Theorem \ref{RMP}). Consequently,  
 the opposite direction 
 is derived and complete the proof of Theorem \ref{RMP1}.

\medskip
As remarked above, the original process $X$ is regarded as a subprocess of $X^{\rm res}$ by  
$\exp(-A^k_t)$. As a result, we obtain the following equations: 
\begin{align}\la{pre-ia}
{\rm (i)}\ \  &\bfP_x(\zeta=\zeta^p<\infty)=\bfE^{\rm res}_x\left(e^{-A^k_{\zeta}};\zeta<\infty\right)\nonumber\\
{\rm (ii)}\ \ &\bfP_x(\zeta=\zeta^i<\infty)=1-\bfE^{\rm res}_x\left(e^{-A^k_{\zeta}}\right)\\
{\rm (iii)}\ \ &\bfP_x(\zeta=\infty)=\bfE^{\rm res}_x\left(e^{-A^k_{\infty}};\zeta=\infty\right)\nonumber
\end{align}
(Corollary \ref{masa}), consequently
 \begin{align}\la{g-c}
1&=\bfP_x(\zeta<\infty)+\bfP_x(\zeta=\infty)\nonumber\\
&=\bfP_x(\zeta=\zeta^p<\infty)+\bfP_x(\zeta=\zeta^i<\infty)+\bfP_x(\zeta=\infty)\\
&=\bfE^{\rm res}_x\left(e^{-A^k_{\zeta}};\zeta<\infty\right)
+\left(1-\bfE^{\rm res}_x\left(e^{-A^k_{\zeta}}\right)\right)+
\bfE^{\rm res}_x\left(e^{-A^k_{\infty}};\zeta=\infty\right).\nonumber
\end{align} 

Keller and Lenz \cite[Definition 1.1]{KL} define the  {\it stochastic completeness at infinity} (SCI for short) of 
Dirichlet forms
$(\cE,\cF)$ on graphs (or the corresponding symmetric 
Markov process $X$). In \cite[Section 7.9]{KLW}, they give a probabilistic interpretation of this 
concept, that is, a probabilistic characterization of {stochastic completeness at infinity} is equivalent to that 
\bequ\la{ngc-0}
\bfE^{\rm res}_x\left(e^{-A^k_{\zeta}};\zeta<\infty\right)=0.
\eequ
Masamune and Schmidt \cite{MS} call this property the {\it generalized conservation property} 
  and
  extend Khasminskii’s criterion for the generalized conservation property. Moreover, they 
  prove the equivalence
with the ordinary conservation property of time changed processes (\cite[Theorem 3.5]{MS}).
We see from (\ref{pre-ia})(i) that (SCI) is also expressed as 
\bequ\la{ngc}
 ({\bf SCI})\qquad \ \ \bfP_x(\zeta=\zeta^p<\infty)=0,\qquad \ \ 
\eequ
which make the image of (SCI) concrete probabilistically. 
As a result of (\ref{ngc-0}), we see that if $X^{\rm res}$ is conservative, 
$\bfP^{\rm res}_x(\zeta=\infty)=1$, then $X$ has (SCI). 
If $X$ itself is conservative, $\bfP_x(\zeta=\infty)=1$, 
then the killing part disappears and $X^{\rm res}$ is identified with $X$, 
consequently $X$ has (SCI). 

We see that the next equivalences hold: 
\begin{align}\la{gcp-eq}
\bfP_x(\zeta=\zeta^i<\infty)=1\Longleftrightarrow\bfE^{\rm res}_x\left(e^{-A^k_{\zeta}}\right)=0\Longleftrightarrow\bfP^{\rm res}_x\left({A^k_{\zeta}}=\infty\right)=1.
\end{align}
When we denote by $X^{{\rm res},k}$ the time changed process of 
$X^{\rm res}$ by $A^k_t$,  
${A^k_{\zeta}}$ is regarded as the life time of $X^{{\rm res},k}$ 
by \cite[(65.2)]{Sh}, and thus $\bfP^{\rm res}_x({A^k_{\zeta}}=\infty)=1$
 implies the
conservativeness of $X^{{\rm res},k}$. 

Here we define the {\it explosion by killing} (({EK}) for short) by
\bequ
 ({\bf EK})\qquad \ \ \bfP_x(\zeta=\zeta^i<\infty)=1.\qquad \ \ 
\eequ
(EK) is equivalent to (SCI) and $\bfP_x(\zeta<\infty)=1$.
We can conclude from the discussion above that (EK) of $X$ is equivalent to 
the conservativeness of the time changed process of $X^{\rm res}$ 
by $A^k_t$. 
Moreover,  we prove in Theorem \ref{Li} and Remark \ref{exten} that (EK)  
is also equivalent to the following Liouville type theorem: 
A function $h\in{\cS}^{\rm sub}\cap{\cS}^{\rm sup}$, then $h(x)=0$ 
for q.e. $x$. Here 
${\cS}^{\rm sub}={\cS}^{0,\rm sub}$ (${\cS}^{\rm sup}={\cS}^{0,\rm sup}$).

For $\lambda\ge 0$ define $\cL^{(\lambda)}$ as the self-adjoint 
operator associated with the closed symmetric form 
$(\cE_\lambda(:=\cE+\lambda(\cdot,\cdot)_m),\cF)$, $(-\cL^{(\lambda)} u,v)_m=\cE_\lambda(u,v)$. 
We can extend the spaces, $\cS^{\rm sub}$ and $\cS^{\rm sup}$ or the concepts,
(SCP) and (EK) to those associated 
with $(\cE_\lambda,\cF)$, and denote these by $\cS_\lambda^{\rm sub}$ and $\cS_\lambda^{\rm sup}$ or 
(SCP$_\lambda$) and (EK$_\lambda$). Noting that for $\lambda>0$
\begin{align*}
\bfE^{\rm res}_x\left(e^{-A^k_{\zeta}};\zeta<\infty\right)=0&\ \Longleftrightarrow\ 
\bfE^{\rm res}_x\left(e^{-A^k_{\zeta}-\lambda\zeta};\zeta<\infty\right)=0\\
&\ \Longleftrightarrow\ \bfE^{\rm res}_x\left(e^{-A^k_{\zeta}-\lambda\zeta}\right)=0,
\end{align*}
we see that (SCP) is equivalent to (SCP$_\lambda)$ or (EK$_\lambda)$ for each $\lambda>0$,
and thus $h\in\cS_\lambda^{\rm sub}\cap\cS_\lambda^{\rm sup}$ equals 0 q.e. if (SCP) holds.
 
 Let $X^D=(\bfP^D_x,X_t,\zeta)$ be the absorbing symmetric 
$\alpha$-stable process ($0<\alpha<2$) on a bounded Lipschitz open set 
$D\subset\bR^d$. Then $\bfP^D_x(\zeta<\infty)=1$ 
and $\bfP^{D,\rm{res}}_x(\zeta=\infty)=1$ for $\alpha\le 1$ 
(\cite[Theorem 1.1]{BBC}). Note that in \cite{BBC} the resurrected process 
$X^{D,{\rm res}}$ of $X^D$
 is called 
a {\it censored stable process}. 
Hence we see that if $\alpha\le 1$, then   
the Liouville property with respect to $X^D$ holds (Example \ref{bog}).

Finally, we would like to make a comment on the strong maximum principle. 
 In our point of view the strong maximum principle to 
the operator $\cL(=\cL^{0})$: If $h\in {\cS}^{\rm sub}\cap C(E)$, then 
$h$ is constant $M=\sup_{x\in E}h(x)$
   	or $h(x)<M$ for all $x\in E$, follows from the irreducibility of the Markov process 
$X$ generated by the regular Dirichlet form $(\cE,\cF)$ (Theorem \ref{BP-1}).  
More precisely, if a symmetric Markov process $\bfP_x$ satisfying the absolute continuity condition (AC) 
(for definition, see Section 2) is irreducible, then for any Borel set $G$ of positive capacity 
$$
\bfP_x(\sigma_G<\zeta)>0\ \  \text{for all}\ x\in E,
$$
where $\sigma_G=\inf\{t>0\mid X_t\in G\}$.

\section{subsolutions and supersolutions}
Let $E$ be a locally compact separable metric space and $m$ a positive
Radon measure on $E$ with full topological support.
Let $(\cE, \cF)$ be a regular symmetric Dirichlet form on
$L^2(E;m)$. 
We denote by $u\in {\cF}_{\loc}$ if for any relatively compact open set $D$ there exists a function $v\in\cF$
such that $u=v$ $m$-a.e. on $D$. 
 Let $X = (\Om, 
 \{\bfP_x\}_{x \in E}, \{X_t\}_{t \ge 0},
\zeta)$ be the symmetric Hunt process generated by $(\cE, \cF)$,
where 
$\zeta $ is the lifetime of $X$.
 Denote by $\{p_t\}_{t\geq 0}$ and  $\{R_\alpha\}_{\alpha \geq 0} $ the semi-group and resolvent 
of $X$: For a bounded Borel function $f$ on $E$
$$
p_tf(x)=\bfE_x(f(X_t);t<\zeta),\ \ \ \ \ R_\alpha f(x)=\int_0^\infty e^{-\alpha t} p_tf(x)dt.
$$
Through this paper we assume that $X$ satisfies next two conditions:  

\bigskip
	{\textbf{Irreducibility (I)}}. \ If a Borel set $A$ is $\{p_t\}_{t\geq 0}$-invariant, i.e., 
      $p_t(1_Af)(x)=1_Ap_tf(x)$\ $m$-a.e. for any $f\in
      L^2(E;m)\cap {\mathscr B}_b(E)$  
      and $t>0$, then $A$ satisfies either $m(A)=0$ or $m(E\setminus A)=0$.

\medskip      
{\textbf{Strong Feller Property (SF)}}. \  For each $t>0$,
      $p_t({\mathscr  B}_b(E))\subset C_b(E)$, 
      where $C_b(E)$ is the space of bounded continuous functions on $E$.
\medskip

We remark that \textbf{(SF)} implies the following condition
(\cite[Theorem 4.2.4]{FOT}).

\medskip
{\textbf{Absolute Continuity Condition (AC)}}. \ The transition probability and resolvent of $X$
 is absolutely continuous with respect to $m$;  for each $t>0,\ \alpha>0$ and $x\in E$
$$
p_tf(x)=\int_Ep(t,x,y)f(y)m(dy),\ \ \ R_\alpha f(x)=\int_Er_\alpha(x,y)f(y)m(dy).
$$

\medskip

We introduce a subspace $\cF^\dagger_{\loc}$ of $\cF_{\loc}$ by
\bequ\la{dag}
\cF^\dagger_{\loc}=\left\{u\in\cF_{\loc}\,\Big|\, \mu^j_{\langle u\rangle}\ 
\text{is a Radon measure on $E$}\ \right\},
\eequ
where $\mu^j_{\langle u\rangle}$ is a positive measure defined in (\ref{mj}).
Denoting by $\cL$ the self-adjoint 
operator associated with the Dirichlet form $(\cE,\cF)$,
$(-\cL u,v)_m=\cE(u,v)$, we call a function $h\in \cF^\dagger_{\loc}\cap L^\infty_{\loc}$ 
a {\it solution} ({\it subsolution}, {\it supersolution}) to
$\cL u=0$ if 
\bequ\la{subsol}
\cE(h,\varphi)=0 \ (\leq 0,\ \geq 0) \ \text{for all}\ \varphi\in\cF_+\cap C_{0}(E).
\eequ
For $u\in \cF^\dagger_{\loc}$ and $\varphi\in\cF_+\cap C_{0}(E)$, $\cE(u,\varphi)$ 
is well-defined, in other words,
$$
\left|\iint_{E\times E}(u(x)-u(y))(\varphi(x)-\varphi(y))J(dx,dy)\right|<\infty,
$$ 
which is shown by the same argument on the jumping part $\cE^j$ in the proof of Lemma \ref{nuh} below.

We write ${\bf S}\,({\bf S}^{\rm sub},\ {\bf S}^{\rm sup})$ for 
the space of {solutions} ({subsolutions}, {supersolutions}) 
and introduce the function spaces:
\begin{align}\la{sol}
\cS^{\rm sub}&= \{h\in {\bf S}^{\rm sub}\mid   \|h^+\|_\infty<\infty\},\\
\cS^{\rm sup}&= \{h\in {\bf S}^{\rm sup}\mid   \|h^-\|_\infty>-\infty\}.
\end{align}
We further introduce the function spaces as follows: 
\begin{align}\la{exc}
\widetilde{\cS}^{\rm sub}&= \{ h\in\cF^\dagger_{\loc}\cap L^\infty_{\loc}\mid  
 \|h^+\|_\infty<\infty,\   p_th \ge h\ {\rm q.e.} \},\\
\widetilde{\cS}^{\rm sup}&= \{ h\in\cF^\dagger_{\loc}\cap L^\infty_{\loc}\mid  
 \|h^-\|_\infty>-\infty,\   p_th \le h\ {\rm q.e.}\}.
\end{align}

By using the argument in the proof 
of \cite[Theorem 5.1]{Mi}, we have the next lemma.

\begin{lem}\label{G-tildeH}
Let $h\in\cF^\dagger_{\loc}\cap L^\infty_{\loc}$ and $\varphi\in\cF\cap C_0(E)$. Let
$K={\rm supp}[\varphi]$ and $G$ a relatively compact open set containing $K$. Let $\{\psi_n\}$ be a sequence in $\cF\cap C_0(E)$
such that $0\le\psi_n\le 1,\ \psi_n=1$ on $G$, $\psi_n\uparrow 1$. Then
$$	
\lim_{n\to\infty}\cE(h\psi_n,\varphi)=\cE(h,\varphi).
$$
\end{lem}
\begin{proof}
On account of $h \psi_n =h$ on $G$, we see from Beurling-Deny formula \cite[Theorem 3.2.1]{FOT} that 
 $\cE(h \psi_n , \varphi)$ is equal to
\begin{align}\la{4}
&\;\frac{1}{2} \int_E d\mu_{\langle h, \varphi \rangle}^c 
+ \iint_{K \times K} (h(x)-h(y))(\varphi(x)-\varphi(y)) J(dx, dy) \nonumber \\
&\hspace{0.5em} + 2\iint_{K \times (K^c \cap G)} (h(x)-h(y))(\varphi(x)-\varphi(y)) J(dx, dy) \\
&\hspace{0.5em} + 2\iint_{K \times G^c} (h(x)-h(y)\psi_n (y))\varphi(x)\, J(dx, dy)
+ \int_E h\varphi\, d\kappa.\nonumber
\end{align} 
Since $J(K\times G^c)<\infty$ and for $h\in\cF^\dagger_{\loc}$ 
$$
\int_E\varphi(x)d\mu^j_{\<h\>}(x)<\infty,
$$
we see  
\begin{align*}
&\ \ \ \  |(h(x)-h(y)\psi_n(y))\varphi(x)|\\
&\leq|h(x)-h(x)\psi_n (y)|\varphi(x)+|h(x)-h(y)||\psi_n (y)\varphi(x)|\\
&\leq|h(x)|\varphi(x)1_{G^c}(y)+|h(x)-h(y)|\varphi(x)\in L^1(K\times G^c;J).
\end{align*}
Hence by the 
 dominated convergence theorem, the fourth term of (\ref{4}) tends to 
\begin{align*}
&
2\iint_{K \times G^c} (h(x)-h(y))\varphi(x)J(dx, dy)\\
&\qquad =2\iint_{K \times G^c} (h(x)-h(y))(\varphi(x)-\varphi(y))J(dx, dy)
\end{align*}  
as $n\to\infty$. Hence 
$$
\cE(h, \varphi)= \lim_{n\to\infty} \cE(h\psi_n, \varphi). 
$$
 \end{proof}

\begin{lem}\label{H-tildeH}
It holds that
$$	
\widetilde{\cS}_-^{\rm{sub}}\subset{\cS}_-^{\rm{sub}},\ \ \ \ \widetilde{\cS}_+^{\rm{sup}}\subset{\cS}_+^{\rm{sup}}.
$$
\end{lem}
\begin{proof}
We only prove that $\widetilde{\cS}_-^{\rm{sub}}\subset{\cS}_-^{\rm{sub}}$. 

Let $\varphi\in\cF_+\cap C_0(E)$ and write $K$ for the support of $\varphi$. Take 
a relatively compact open set $G$ and a sequence 
$\{\psi_n\}\subset \cF_+\cap C_0(E)$ as in Lemma \ref{G-tildeH}.
Then for $h\in\widetilde{\cS}_-^{\rm{sub}}$, 
$\cE(h \psi_n, \varphi)  \le 0$.
 Indeed, since $h\psi_n\in\cF$ and $p_t(h\psi_n) \ge p_th \ge h$, 
\begin{align*}
\cE(h \psi_n, \varphi)&=\lim_{t \downarrow 0} \frac{1}{t} \bigl( h\psi_n - p_t(h\psi_n), \varphi \bigr)_m\\
&\le \varlimsup_{t \downarrow 0} \frac{1}{t} \biggl( \bigl( h, \varphi \bigr)_m 
- \bigl(p_th, \varphi \bigr)_m \bigr)\le 0,
\end{align*}
where $(\ ,\ )_m$ is the inner product of $L^2(E;m)$. 
By Lemma \ref{G-tildeH} we have  
$$
\cE(h, \varphi)= \lim_{n\to\infty} \cE(h\psi_n, \varphi) \le 0. 
$$
 \end{proof}

 	
 	\bl\la{nuh}
 	For  $h\in\cS^{\rm sub}$, there exists 
 	a smooth positive Radon measure $\nu_h$ such that 
 	\bequ\la{me-n}
 	\cE(h,\varphi)=-\int_E\varphi d\nu_h,\ \ \varphi\in \cF\cap C_0(E).
 	\eequ
 	\el
 	 \begin{proof}
 	 	Denote ${\mathcal C}=\cF\cap C_0(E)$ and define a  function $I$  on $\cC$ by
 	 	\begin{equation}\la{i-r}
 	 		I(\varphi)=-\cE(h,\varphi ), \quad \varphi \in{\mathcal C}. 
 	 	\end{equation}	
 	 	We confirm in \cite[Lemma 4.7]{T3} that ${\mathcal C}$
 	 	is a {\it Stone vector lattice} and $I(\varphi )$ is a {\it pre-integral}
 	 	 (For these definitions, see \cite[p.143]{Dud}). We then know 
 	 	from Stone-Daniell theorem \cite[Theorem 4.5.2]{Dud} that there exists a 
 	 	positive Borel measure
 	 	$\nu_h$ such that  
 	 	\begin{equation*}
 	 		I(\varphi )=\int_E\varphi d\nu_h, \quad \varphi \in{\mathcal C}. 
 	 	\end{equation*} 
 	 	We see by the relation (\ref{me-n}) that $\nu_h$ is Radon. Moreover, we see 
 	 	by the same argument as in \cite[Lemma 4.7]{T3}, \cite[Lemma 4.1]{Mi} that 
 	 	the measure $\nu_h$ is smooth. Indeed, let $K$ be a compact set with Cap$(K)	=0$. 
 	 	Take relatively compact open sets $G$ and $D$ such that
 	 	 $K\subset G\subset\bar{G}\subset D\subset
 	 	E$. Let $\{\varphi_n\}$ be a subset of $\cF_+\cap C_0(G)$ such that 
 	 	$\varphi_n\geq 1$ on $K$, and 
 	 	$\lim_{n\to\infty}\cE(\varphi_n)=0$. 
 	 	The existence of such 
 	 	a sequence $\{\varphi_n\}$ follows from 
 	 	\cite[Lemma 2.2.7, Theorem 4.4.3]{FOT}. Take $\psi \in \cF_+\cap C_0(D)$ such that  $\psi\leq 1$ 
 	 	and  $\psi=1$ on $G$. Let $\cE^c,\cE^j,\cE^k$ be the strongly local part, jumping part, killing part of 
 	 	$\cE$ in the Beurling-Deny formula.
 	 	
 	 	We see from the local property of $\cE^c$ and $\cE^k$ that 
 	 	$$|\cE^c(h,\varphi_n)|\le
 	 	\cE^c(h\psi)^{1/2}\cE^c(\varphi_n)^{1/2},\ \ 
 	 	|\cE^k(h,\varphi_n)|\le\cE^k(h\psi)^{1/2}\cE^k(\varphi_n)^{1/2}.
 	 	$$ 	 	
 	 	The jumping part $\cE^j(h,\varphi_n)$ equals 
 	 	\begin{align*}
 	 		\cE^j(h,\varphi_n)=& \iint_{E \times E} (h(x)-h(y))(\varphi_n(x)-\varphi_n(y)) J(dx, dy)\\
 	 		=& \iint_{G \times G} (h(x)-h(y))(\varphi_n(x)-\varphi_n(y)) J(dx, dy)\\
 	 		+&2 \iint_{G \times G^c} (h(x)-h(y))(\varphi_n(x)-\varphi_n(y)) J(dx, dy).
 	 	\end{align*}
 	 	The first term of the right hand side equals
 	 	$$
 	 	\iint_{G \times G} (h(x)\psi(x)-h(y)\psi(y))(\varphi_n(x)-\varphi_n(y)) J(dx, dy)
 	 	\le\cE^j(h\psi)^{1/2}\cE^j(\varphi_n)^{1/2}
 	 	$$
 	 	and the second term is less than $2\mu^j_{\<h\>}(G)^{1/2}\cE^j(\varphi_n)^{1/2}$.
 	 	Hence noting $h\psi\in\cF$, we have 
 	 	\begin{align*}
 	 	\nu_h(K)&\leq\int_E\varphi_nd\nu_h=-\cE(h,\varphi_n)
 	 	=-\left(\cE^c(h,\varphi_n)+\cE^j(h,\varphi_n)+\cE^k(h,\varphi_n)\right)\\
 	 	&\leq\cE^c(h\psi)^{1/2}\cE^c(\varphi_n)^{1/2}
 	 	+\cE^j(h\psi)^{1/2}\cE^j(\varphi_n)^{1/2}+2\mu^j_{\<h\>}(G)^{1/2}\cE^j(\varphi_n)^{1/2}\\
 	 	&\ \ \  +\cE^k(h\psi)^{1/2}\cE^k(\varphi_n)^{1/2}\\ 
 	 	&\le\left(\cE(h\psi)^{1/2}+2\mu^j_{\<h\>}(G)^{1/2}\right)
 	 	\cE(\varphi_n)^{1/2}\longrightarrow 0\ \ \text{as}\ n\to\infty.
 	 	\end{align*}
  Therefore the measure $\nu_h$ is smooth.
 	 \end{proof}	
 	
\br\la{r-s} \rm 
Let $h\in\cS^{\rm sub}$ and 
$h^M=h-M$, $M:=\|h^+\|_\infty$. Noting that the constant function M 
belongs to $\cF^\dagger_{\loc}\cap L^\infty_{\loc}$ and that 
$$
\cE(h^M,\varphi)=\cE(h,\varphi)-M\cE(1,\varphi)=\cE(h,\varphi)-M\int_E\varphi dk\le 0,
$$
we see that $h^M\in\cS^{\rm sub}_-$. Since  
$$
\cE(h,\varphi)=-\int_E\varphi (d\nu_{h^M}-Mdk),
$$
$\nu_{h}$ equals $\nu_{h^M}-Mk$.  
$\nu_h$ looks signed, but it's actually positive.

\er

 		Let $\{K_n\}$ be a sequence of increasing compact sets such that 
 		$K_1\subset \mathring{K}_{2}\subset\cdots\subset K_n\subset \mathring{K}_{n+1}\subset\cdots$ and 
$K_n\uparrow E$, where $\mathring{K}$ is the interior of $K$. Denote by $\tau_n$ the first exit time from $K_n$, 
\bequ\la{exit}
\tau_n=\inf\{t>0\mid X_t\not\in K_n\}.
\eequ	
 	The next lemma is an extension of \cite[Lemma 3.18, Lemma 3.19]{T4}, 
 	where a function in $\cF_e$ is treated.

 	\bl\la{invariant-0}
 	For $h\in\cS^{\rm sub}$ there exists a sequence $\{\sigma_n\}$ of stopping times such that  
 		$\sigma_n<\zeta$, $\sigma_n\uparrow\zeta$ and 
 		\begin{align}\la{s}
 			h(x)\le\bfE^{\rm res}_x\left(e^{-A^k_{t\wedge \sigma_n}}h(X_{t\wedge \sigma_n})\right) 
 			\ \text{\rm q.e.}\ x .
 		\end{align} 	 
 	\el
 	\begin{proof}
 	  		 Let $(\cE^{\rm res},\cF^{\rm res})$ 
  		be the {\it resurrected Dirichlet form} defined in \cite[(5.2.25)]{CF}, which is 
 		a regular Dirichlet form on $L^2(E;m)$. We see from \cite[Theorem 5.2.17]{CF} 
 		that $\cF^{\rm res}\supset\cF$ and 
 		$\cF^{\rm res}\cap L^2(E;k)=\cF$. 
 		As a result, we know that  $\cF^{\rm res}\cap C_0(E)=\cF\cap C_0(E)$, $({\cF^
 		{\rm res}}^{\dagger})_{\loc}\supset\cF^\dagger_{\loc}$ and that	 for $h\in{\cS}^{\rm{sub}}$
 		$$
 		\cE^{\rm res}(h,\varphi )=-\int_X\varphi (hdk+d\nu_h), \ \ \varphi \in \cF^{\rm res}\cap C_0(E).
 		$$
 		Since $k$ is a smooth Radon measure, so is $hk$.

 		Let $X^{\rm res}=
 		(\bfP_x^{\rm res},X_t,\zeta)$ be the {\it resurrected process} of $X$, that is, 
 		the Hunt process generated by $(\cE^{\rm res},\cF^{\rm res})$
 		(cf. \cite[Theorem 5.2.17]{CF}). We then see from \cite[Corollary 3.3]{Mi} that 
 		$$
 		h(X_t)=h(X_0)+M^{[h]}_t+\int_0^th(X_s)dA^{k}_s+A^{\nu_h}_t, \ \ t<\zeta,\ \ 
 		\bfP^{\rm res}_x\text{-a.s. q.e.}\ x ,
 		$$
 		where $M^{[h]}_t\in \cM_{\loc}^{[[0,\zeta[[}$ is the martingale part of the 
 		Fukushima decomposition, i.e., there
 		exist a sequence $\{S_n\}$ of stopping times with $S_n\uparrow\zeta$ and 
 		a sequence $\{M^n_t\}$ of square integrable martingale AFs  such that 
 \bequ\la{ku-fu}
 		M^{[h]}_{t\wedge S_n}1_{\{t\wedge S_n<\zeta\}}=M^{n}_{t\wedge S_n}1_{\{t\wedge S_n<\zeta\}}
 	\eequ	
 		(\cite[Theorem 4.2]{Ku}). 

 		Note that $X_{\zeta}=X_{\zeta-}(=\infty)$, $\bfP^{\rm res}_x$-a.s. on $\{\zeta<\infty\}$
 		because $(\cE^{\rm res},\cF^{\rm res})$ has no killing
 		part (\cite[Theorem 5.3.1]{FOT}). Consequently, 
 		$\tau_n<\zeta$,  $\tau_n\uparrow\zeta$ on $\{\zeta<\infty\}$.
 		Define $\sigma_n=S_n\wedge \tau_n$. Then 
 		$\sigma_n<\zeta$ and $\sigma_n\uparrow\zeta$. 
 		We see from It$\hat{\rm o}$'s formula that 
 		\begin{align}\la{both}
 			&e^{-A^k_{t\wedge \sigma_n}}h(X_{t\wedge \sigma_n})
 			=h(X_0)
 			-\int_0^{t\wedge \sigma_n}e^{-A^k_s}h(X_s)dA^k_s
 			+\int_0^{t\wedge \sigma_n}e^{-A^k_s}dM_s^{[h]}\nonumber \\
 			&\qquad \ \ \ +\int_0^{t\wedge \sigma_n}e^{-A^k_s}h(X_s)dA^k_s
 			+\int_0^{t\wedge \sigma_n}e^{-A^k_s}dA^{\nu_h}_s\nonumber	\\
 			&\qquad \ \ =h(X_0)+\int_0^{t\wedge \sigma_n}e^{-A^k_s}dM_s^{[h]}
 			+\int_0^{t\wedge \sigma_n}e^{-A^k_s}dA^{\nu_h}_s
 			, \ \ \bfP^{\rm res}_x\text{-a.s. q.e. } x .
 		\end{align}
 		Noting that $\int_0^{t\wedge \sigma_n}e^{-A^k_s}dM_s^{[h]}$ is 
 		a square integrable martingale by (\ref{ku-fu}) and 
 		$\int_0^{t\wedge \sigma_n}e^{-A^k_s}dA^{\nu_h}_s\ge 0$,  
 	 		we have this lemma by taking expectation of both sides of (\ref{both}).
  \end{proof}

 	\bl\la{invariant}
 	It holds that
$$	
{\cS}_-^{\rm{sub}}\subset\widetilde{\cS}_-^{sub},\ \ \ \ {\cS}_+^{\rm{sup}}\subset\widetilde{\cS}_+^{sup}.
$$
 	\el
 	\begin{proof}
 	We will only prove that ${\cS}_-^{\rm{sub}}\subset\widetilde{\cS}_-^{\rm{sub}}$. 

 		On account of $h\le0$, we see from Fatou's lemma that 
 	
 		\begin{align}\la{eta}
 			\bfE_x^{\rm res}\left(e^{-A^{k}_{t}}h(X_{t})\right)&
 			=\bfE_x^{\rm res}\left(e^{-A^{k}_{t}}
 			h(X_{t}); t<\zeta\right)\nonumber  \\
 			&=-\bfE_x^{\rm res}\left(\varliminf_{n\to\infty}\left(e^{-A^{k}_{t\wedge \sigma_n}}
 			(-h)(X_{t\wedge \sigma_n})\right); t<\zeta\right)\nonumber  \\
 			&\ge-\bfE^{\rm res}_x\left(\varliminf_{n\to\infty}\left(e^{-A^k_{t\wedge \sigma_n}}(-h)(X_{t
 			\wedge \sigma_n})\right)\right)\\
 			&\ge-\varliminf_{n\to\infty}\left(-\bfE^{\rm res}_x\left(e^{-A^k_{t\wedge \sigma_n}}h(X_{t
 			\wedge \sigma_n})\right)\right)\ \ {\rm q.e.}\ x, \nonumber
 		\end{align}
 	and thus 
 $$		
 	h(x)\le\bfE_x^{\rm res}\left(e^{-A^{k}_{t}}h(X_{t})\right)=p_th(x) \ \ \text{q.e. } x.
$$ 
by (\ref{s}). 
 		Applying \cite[Theorem 5.2.17]{CF} again, 
 		we see that the left hand side of (\ref{eta}) is equal to 
 		$
 		\bfE_x\left(h(X_t)\right) =p_th(x).
 		$
 	\end{proof}	
 

By Lemma \ref{H-tildeH} and Lemma \ref{invariant}, we have 

   \begin{thm}\la{eq}
 It holds that
 $$
 {\cS}_-^{\rm{sub}}=\widetilde{\cS}_-^{\rm{sub}},\ \ \ \
 {\cS}_+^{\rm{sup}}=\widetilde{\cS}_+^{\rm{sup}}.
 $$ 
       \end{thm}		
  	

\medskip
Let $\{\tau_n\}$ be a sequence of stopping times defined in (\ref{exit}). We define  
\begin{align}\la{p-set}
\Omega^p&=\cap_{n=1}^\infty\{\tau_n<\zeta<\infty\},
\\ 
\Omega^{i}&=(\Omega^{p})^c=\left(\cup_{n=1}^\infty\{\tau_n=\zeta<\infty\}\right)
\cup\{\zeta=\infty\}.
\end{align}
Let $\zeta^p$ (resp. $\zeta^i$) be the predictable (resp. totally inaccessible) part of $\zeta$, that is, 
\begin{equation}\la{dualp}
\zeta^p=\left\{
\begin{split}
& \zeta, \ \ \ \ \,\omega\in \Omega^p,  \\
& \infty,\ \ \ \omega\in \Omega^i,
\end{split}
\right.\qquad 
\zeta^i=\left\{
\begin{split}
& \zeta, \ \ \ \ \,\omega\in \Omega^i,  \\
& \infty,\ \ \ \omega\in \Omega^p.
\end{split}
\right.
\end{equation}
Then $\zeta$ is written as $\zeta=\zeta^p\wedge\zeta^i$. 
For the decomposition of the predictable part and the totally inaccessible part, refer \cite[Lemma (13.4)]
{RW}. Note that predictability is equivalent to accessibility in case when $X$ is a Hunt process (\cite[Theorem 15.1]{RW}).

\bl \la{k-shl}
For $u\in C(E_\infty)$
\bequ\la{k-sh}
\bfE_x(u(X_{\zeta-});\zeta=\zeta^p<\infty)
=\bfE^{\rm res}_x\left(e^{-A^{k}_{\zeta}}u(X_{\zeta-});\zeta<\infty\right).
\eequ
Here $C(E_\infty)$ is the set of continuous functions on the one-point compactification $E_\infty$ 
of $E$.
\el
\begin{proof} Let $\{\tau_n\}$ be a sequence of stopping times in (\ref{exit}).
By \cite[(62.13)]{Sh}, 
\begin{align*}
&\ \ \ \ \bfE_x\left(u(X_{\tau_n})1_{\{\tau_n<\zeta<\infty\}}\right)\\
&=\bfE_x^{\rm res}\left(\int_0^{\tau_n}u(X_{\tau_n}(k_t))
1_{\{\tau_n(k_t)<\zeta(k_t)<\infty\}}\left(-de^{-A^k_t}\right)\right.\\
&\qquad \qquad \qquad \qquad \qquad \qquad \qquad \left. +e^{-A^k_{\tau_n}}u(X_{\tau_n})1_{\{\tau_n<\zeta<\infty\}}\right).
\end{align*}
Here $k_t,\ 0\le t\le\infty$ is the killing operators (For definition, see \cite[Definition (11.3)]{Sh}).
Since 
$$
\tau_n(k_t)<\zeta(k_t)\Longleftrightarrow\tau_n\wedge t<\zeta\wedge t\Longleftrightarrow\tau_n<t,
$$
by \cite[Proposition (11.11)]{Sh}, the right hand side equals
$$
\bfE_x^{\rm res}\left(e^{-A^k_{\tau_n}}u(X_{\tau_n})1_{\{\tau_n<\zeta<\infty\}}\right).
$$
Note that $\bfP^{\rm res}_x(\zeta<\infty)=1$ implies $\bfP_x(\zeta<\infty)=1$. Hence on account of  
 $\bfP^{\rm res}_x(\tau_n<\zeta<\infty)=1$,  
we have
$$
\bfE_x(u(X_{\tau_n});\tau_n<\zeta<\infty)
=\bfE^{\rm res}_x\left(e^{-A^{k}_{\tau_n}}u(X_{\tau_n});\zeta<\infty\right),
$$
which implies (\ref{k-sh}) by letting $n\to\infty$.
\end{proof}

\bc\la{masa}
{\ \rm (i)} $\bfP_x(\zeta=\zeta^p<\infty)=
\bfE^{\rm res}_x\left(e^{-A^k_{\zeta}};\zeta<\infty\right)$.\\
{\ \rm (ii)} $\bfP_x(\zeta=\zeta^i<\infty)=1-\bfE^{\rm res}_x\left(e^{-A^k_{\zeta}}\right)$.\\
{\rm (iii)} $\bfP_x(\zeta=\infty)=\bfE^{\rm res}_x\left(e^{-A^k_{\infty}};\zeta=\infty\right).$
\ec
\begin{proof}
The claim (i) follows from Lemma \ref{k-shl} by taking $u\equiv1$.
Since
\begin{align*}
\bfP_x(\zeta=\infty)&=\bfE^{\rm res}_x
\left(\int_0^\zeta 1_{\{\zeta(k_t)=\infty\}}\left(-de^{-A^k_{t}}\right)+
e^{-A^k_{\zeta}}1_{\{\zeta=\infty\}}\right)\\
&=\bfE^{\rm res}_x\left(e^{-A^k_{\infty}};\zeta=\infty\right)
\end{align*}
by $\zeta(k_t)=\zeta\wedge t$, 
the claim (iii) follows. Noting that 
\begin{align*}
\bfP_x(\zeta=\zeta^p<\infty)+\bfP_x(\zeta=\infty)&=\bfE^{\rm res}_x
\left(e^{-A^k_{\zeta}};\zeta<\infty\right)+\bfE^{\rm res}_x
\left(e^{-A^k_{\infty}};\zeta=\infty\right)\\
&=\bfE^{\rm res}_x\left(e^{-A^k_{\zeta}}\right),
\end{align*}
we have 
\begin{align*}
\bfP_x(\zeta=\zeta^i<\infty)&=1-\left(\bfP_x(\zeta=\zeta^p<\infty)
+\bfP_x(\zeta=\infty)\right)\\
&=1-\bfE^{\rm res}_x\left(e^{-A^k_{\zeta}}\right).
\end{align*}
\end{proof}
Since
\begin{align*}
&\quad \ \bfP_x(\zeta<\infty)+\bfP_x(\zeta=\infty)\\
&=\bfP_x(\Omega^p)+\bfP_x(\Omega^i\cap\{\zeta<\infty\})+
\bfP_x(\Omega^i\cap\{\zeta=\infty\})\\
&=\bfP_x(\zeta=\zeta^p<\infty)+\bfP_x(\zeta=\zeta^i<\infty)
+\bfP_x(\zeta=\infty),
\end{align*}
according to Corollary \ref{masa} 
\begin{align}\la{g-c}
\bfE^{\rm res}_x\left(e^{-A^k_{\zeta}};\zeta<\infty\right)
+\left(1-\bfE^{\rm res}_x\left(e^{-A^k_{\zeta}}\right)\right)+
\bfE^{\rm res}_x\left(e^{-A^k_{\infty}};\zeta=\infty\right)=1.
\end{align}

Since 
\begin{align*}
-\int_0^t1_E(X_s)e^{-A^k_s}dA^k_s&=\int_0^t1_E(X_s)d(e^{-A^k_s})\\
&=e^{-A^k_t}1_E(X_t)-1-\int_0^te^{-A^k_s}d(1_E(X_s))\\
&=e^{-A^k_t}1_E(X_t)-1+e^{-A^k_\zeta}1_{\{\zeta\le t\}},
\end{align*}
we have 
\begin{align}\la{kl}
1&=\bfE_x^{\rm res}\left(e^{-A^k_t}1_E(X_t)\right)+\bfE_x^{\rm res}\left(\int_0^t1_E(X_s)e^{-A^k_s}dA^k_s\right)+\bfE_x^{\rm res}\left(e^{-A^k_\zeta};\zeta\le t\right)\\
&=\bfE_x^{\rm res}\left(e^{-A^k_t};t<\zeta\right)+\bfE_x^{\rm res}\left(\int_0^{\zeta\wedge t}e^{-A^k_s}dA^k_s\right)+\bfE_x^{\rm res}\left(e^{-A^k_\zeta};\zeta\le t\right).\nonumber
\end{align}
Keller and Lenz \cite{KL} define the  {\it stochastic completeness at infinity} ((SCI) in 
abbreviation) of 
$(\cE,\cF)$ (or the corresponding symmetric 
Markov process $X$) by 
\begin{align}\la{kl1}
1&=\bfE_x^{\rm res}\left(e^{-A^k_t};t<\zeta\right)+\bfE_x^{\rm res}\left(\int_0^{\zeta\wedge t}e^{-A^k_s}dA^k_s\right)\nonumber \\
&=\bfE_x^{\rm res}\left(e^{-A^k_t}1_E(X_t)\right)+
\bfE_x^{\rm res}\left(\int_0^t1_E(X_s)e^{-A^k_s}dA^k_s\right), \ t>0.
\end{align}
Note that if $dk=Vdm$, then $A^k_t=\int_0^tV(X_s)ds$ and the equation (\ref{kl1}) is written as
$$
1=p_t1+\int_0^tp_sVds, \ t>0.
$$
By (\ref{kl}), the equality (\ref{kl1}) is equivalent to
$$
\bfE_x^{\rm res}\left(e^{-A^k_\zeta};\zeta\le t\right)=0, \ t>0,
$$
which is equivalent to
\bequ\la{kl-c1}
\bfE_x^{\rm res}\left(e^{-A^k_\zeta};\zeta<\infty\right)=0
\eequ 
because 
$$
\bfE_x^{\rm res}\left(e^{-A^k_\zeta};\zeta\le t\right)\uparrow 
\bfE_x^{\rm res}\left(e^{-A^k_\zeta};\zeta<\infty\right)\ \ \text{as}\ t\to\infty.
$$
Moreover, by the irreducibility of $X^{\rm res}$, it satisfies 
i) $\bfP^{\rm res}_x(\zeta<\infty)=0$ for all $x\in E$ or 
ii) $\bfP^{\rm res}_x(\zeta<\infty)>0$ for all $x\in E$. For the case i), 
(\ref{kl-c1}) holds and for the case ii),
 (\ref{kl-c1})is equivalent to
\bequ\la{kl-c2}
e^{-A^k_\zeta}=0\ \ \bfP^{\rm res}_x\text{-a.s. on}\ \{\zeta<\infty\}\ \Longleftrightarrow
\ \bfP^{\rm res}_x(A^k_\zeta=\infty\,|\,\zeta<\infty)=1.
\eequ 
Therefore, we see that (SCI) is equivalent to that i) $\bfP^{\rm res}_x(\zeta=\infty)=1$
 or $\bfP^{\rm res}_x(A^k_\zeta=\infty|\zeta<\infty)=1$ (\cite[Theorem 7.33]{KLW}). 
 Here we would like to emphasis that using the
concepts of predictable part and totally inaccessible part of the life time, we can 
make the image of (SCI) concrete probablistically, that is,   
\bequ\la{ngc}
({\rm SCI})\ \Longleftrightarrow\ \bfP_x(\zeta=\zeta^p<\infty)=0.
\eequ

 Masamune and Schmidt \cite{MS} call this property the {\it generalized conservation property} 
 ((GCP) for short) and
  extend Khasminskii’s criterion for the generalized conservation property
  and prove the equivalence
with the ordinary conservation property of time changed processes.
We see that if $\bfP^{\rm res}_x(\zeta=\infty)=1$, that is, $X^{\rm res}$ is conservative, 
then $X$ has (SCI). In particular, if $X$ is conservative, $\bfP_x(\zeta=\infty)=1$, 
then the killing part disappear and $X^{\rm res}$ is identified with $X$, consequently, $X$ has the SCI.

As stated in (\ref{kl-c2}), if $\bfP_x(\zeta<\infty)=1$, then (SCI) is equivalent to 
\begin{align}\la{eqi}
\bfP_x(\zeta=\zeta^i<\infty)=1&\ \Longleftrightarrow\ \bfE^{\rm res}_x\left(e^{-A^k_{\zeta}}\right)=0\\
&\ \Longleftrightarrow\ \bfP^{\rm res}_x\left({A^k_{\zeta}}=\infty\right)=1.\nonumber
\end{align}
Let $X^{{\rm res},k}$ be the time changed process of $X^{\rm res}$ 
by $A^k_t$. Note that $A^k_\zeta$ is nothing but the life time of $X^{{\rm res},k}$ (\cite[(65.2)]{Sh}). 
 We then see that if $X^{{\rm res},k}$ is conservative, $\bfP^{\rm res,k}_x(\zeta=\infty)=1$,
 then $\bfP^{\rm res}_x(A^k_\zeta=\infty)=1$. On account of the equivalence above, we define 
 a concept on stochastic incompleteness, {\it explosion by killing} (({EK}) for short) by
\bequ
 ({\rm EK})\qquad \ \ \bfP_x(\zeta=\zeta^i<\infty)=1.\qquad \ \ 
\eequ
(EK)  implies that the Hunt process $X$ explodes by jumping from inside of $E$ to the cemetery 
point $\Delta$ almost surely.
We then have
  
 \bt\la{ma-sh}
The following statements are equivalent:
 
 \smallskip
 \noindent
 {\ \ \rm (i)}  $\bfP_x(\zeta=\infty)=0$ and 
  {\rm(SCI)}.\\ 
{\  \rm (ii)}   {\rm(EK)}\\
{\rm (iii)} $\bfE^{\rm res}_x\left(e^{A^k_\zeta}\right)=0$.\\ 
{\rm (iv)} $\bfP^{\rm res}_x(A^k_\zeta=\infty)=1$.\\ 
{\ \rm (v)} The time changed process of $X^{\rm res}$ by $A^k_t$ is conservative.
 \et

\br
Since $X$ is  the subprocess of $X^{\rm res}$ by $\exp(-A^k_t)$,  
\begin{align*}
&\bfP_x(\tau_n<\zeta\le t)\\
&=\bfE_x^{\rm res}\left(\int_0^{\zeta}
1_{\{\tau_n(k_s)<\zeta(k_s)\le t\}}\left(-de^{-A^k_s}\right)\right)
+\bfE_x^{\rm res}\left(e^{-A^k_{\zeta}};\zeta\le t\right)\\
&=\bfE_x^{\rm res}\left(\int_{\tau_n}^{\zeta}
\left(-de^{-A^k_s}\right)\right)+\bfE_x^{\rm res}
\left(e^{-A^k_{\zeta}};\zeta\le t\right).
\end{align*}
The left and right hand side tend to
$\bfP_x(\zeta=\zeta^p\le t)$ and 
$\bfE^{\rm res}_x\left(e^{-A^k_{\zeta}};\zeta\le t\right)$ respectively 
  by letting $n\to\infty$, and thus 
\bequ\la{p-t}
\bfP_x(\zeta=\zeta^p\le t)=\bfE^{\rm res}_x\left(e^{-A^k_{\zeta}};\zeta\le t\right).
\eequ
Since
\begin{align*}
\bfP_x(\zeta=\zeta^i\le t)&=1-\bfP_x(\zeta>t)-\bfP_x(\zeta=\zeta^p\le t)\\
&=1-\bfE_x^{\rm res}\left(e^{-A^k_{t}};\zeta>t\right)-\bfE^{\rm res}_x\left(e^{-A^k_{\zeta}};\zeta\le t\right),
\end{align*}
we have 
\bequ\la{i-t}
\bfP_x(\zeta=\zeta^i\le t)=1-\bfE^{\rm res}_x
\left(e^{-A^k_{\zeta\wedge t}}\right).
\eequ
Corollary \ref{masa} is regarded as identities obtained by limiting 
equations (\ref{p-t}) and  (\ref{i-t}) as $t\to\infty$.

\er

\br\la{ess-re}
For $\lambda>0$, denote by (SCI$_{\lambda})$(resp. (EK$_{\lambda}$)) the stochastic completeness at infinity 
(resp. explosion by killing) of 
$(\cE_\lambda(=\cE+\lambda(\ ,\ )),\cF)$. Then we see that 
\begin{align*}
\text{(SCI)}\ \bfE^{\rm res}_x\left(e^{-A^k_{\zeta}};\zeta<\infty\right)=0&\Longleftrightarrow \ \text{(SCI$_{\lambda}$)}
\  \bfE^{\rm res}_x\left(e^{-A^k_{\zeta}-\lambda\zeta};\zeta<\infty\right)=0\\
&\Longleftrightarrow  \ \text{(EK$_{\lambda}$)}\,\ \bfE^{\rm res}_x\left(e^{-A^k_{\zeta}-\lambda\zeta}\right)=0.
\end{align*}
Hence, (SCI) is equivalent to (SCI$_{\lambda}$) and (EK$_{\lambda}$) for any $\lambda>0$.
\er

\bl\la{f-eq}
For $h\in{\cS}^{\rm sub}$
\begin{align*}\la{fu-rq}
&h(x)\le \bfE^{\rm res}_x\left(e^{-A^{k}_{t}}h(X_{t});t<\zeta\right)\\
&+\varlimsup_{n\to\infty}\bfE^{\rm res}_x\left(e^{-A^{k}_{\sigma_n}}
h^+(X_{\sigma_n});t\ge\zeta\right)-\varliminf_{n\to\infty}\bfE^{\rm res}_x\left(e^{-A^{k}_{\sigma_n}}
h^-(X_{\sigma_n});t\ge\zeta\right) {\rm q.e.}
\end{align*}
\el
\begin{proof}
On account of (\ref{s}) we have for $h\in{\cS}^{\rm sub}$
$$
 		h(x)\le\bfE^{\rm res}_x\left(e^{-A^k_{t\wedge \sigma_n}}h(X_{t\wedge \sigma_n})\right)\ \ {\rm q.e.}\ 
$$ 	
Since 
\begin{align*}
&\varlimsup_{n\to\infty}\bfE^{\rm res}_x\left(e^{-A^{k}_{t\wedge \sigma_n}}
h(X_{t\wedge \sigma_n})\right)\nonumber \\
\le&\,\varlimsup_{n\to\infty}\left(\bfE^{\rm res}_x\left(e^{-A^{k}_{t\wedge \sigma_n}}
h^+(X_{t\wedge \sigma_n});t<\zeta\right)
+\bfE^{\rm res}_x\left(e^{-A^{k}_{\sigma_n}}h^+(X_{\sigma_n});t\ge\zeta\right)\right)\nonumber\\
&-\varliminf_{n\to\infty}\left(\bfE^{\rm res}_x\left(e^{-A^{k}_{t\wedge \sigma_n}}
h^-(X_{t\wedge \sigma_n});t<\zeta\right)
+\bfE^{\rm res}_x\left(e^{-A^{k}_{\sigma_n}}h^-(X_{\sigma_n});t\ge\zeta\right)\right)\nonumber\\
\le&\ \bfE^{\rm res}_x\left(e^{-A^{k}_{t}}h(X_{t});t<\zeta\right)
+\varlimsup_{n\to\infty}\bfE^{\rm res}_x\left(e^{-A^{k}_{\sigma_n}}
h^+(X_{\sigma_n});t\ge\zeta\right)\\
&\ \ \ -\varliminf_{n\to\infty}\bfE^{\rm res}_x\left(e^{-A^{k}_{\sigma_n}}
h^-(X_{\sigma_n});t\ge\zeta\right)
\end{align*}
the proof is completed.
\end{proof}

\bc\la{f-eq1}
Let $h\in{\cS}^{\rm sub}\cap C(E_\infty)$. Then 
\bequ\la{fu-rq}
h(x)\le \bfE^{\rm res}_x\left(e^{-A^{k}_{t}}h(X_{t});t<\zeta\right)
+h(\infty)\bfE^{\rm res}_x\left(e^{-A^{k}_{\zeta}};t\ge\zeta\right)\ \ {\rm q.e.}
\eequ
\ec

\bl\la{con-res}
If 
$\bfP^{\rm res}_x(\zeta=\infty)=1$, 
then {for} $h\in{\cS}^{\rm sub}$
$$
h(x)\le p_th(x)\ \ {\rm q.e.}
$$
In particular, ${\cS}^{\rm sub}\subset\widetilde{\cS}^{\rm sub}$.
\el
\begin{proof}
If $\bfP^{\rm res}_x(\zeta=\infty)=1$, the expectations on $\{t\ge\zeta\}$ in 
the proof of Lemma \ref{f-eq} disappear. Hence
$$
h(x)\le\bfE^{\rm res}_x\left(e^{-A^{k}_{t}}h(X_{t})\right)=p_th(x)\ \ {\rm q.e.}
$$
\end{proof}

\bp\la{ex-res}
Let $h\in{\cS}^{\rm sub}\cap C(E_\infty)$. If $\bfP_x(\zeta<\infty)=1$, then
\bequ\la{ex-1}
h(x)\le h(\infty)\bfE^{\rm res}_x\left(e^{-A^{k}_{\zeta}};\zeta<\infty\right)
(=h(\infty)\bfP_x(\zeta=\zeta^p<\infty))\ \ {\rm q.e.}
\eequ
\ep
\begin{proof}
 The right hand side of (\ref{fu-rq}) tends to the right hand side of (\ref{ex-1})
 as $t\to\infty$. Indeed, 
\begin{align*}
\varlimsup_{t\to\infty}\bfE^{\rm res}_x\left(e^{-A^{k}_{t}}h(X_{t});t<\zeta\right)&\le
\|h^+\|_\infty\bfE^{\rm res}_x\left(e^{-A^{k}_{\infty}};\zeta=\infty\right)\\
&=\|h^+\|_\infty\bfP_x(\zeta=\infty)=0
\end{align*}
and 
$$
\lim_{t\to\infty}\bfE^{\rm res}_x\left(e^{-A^{k}_{\zeta}};t\ge\zeta\right)=\bfE^{\rm res}_x\left(e^{-A^{k}_{\zeta}};\zeta<\infty\right).
$$
\end{proof}

\begin{prop}\la{p-l}
 If $\bfP_x(\zeta=\zeta^i<\infty)=1$, then every $h\in{\cS}^{\rm sub}$ is non-positive for 
  {\rm q.e.}\ x.
\end{prop}
\begin{proof}
We see from Lemma \ref{f-eq} that 
\begin{align*}
&h(x)\le
\|h^+\|_\infty\left(\bfE^{\rm res}_x\left(e^{-A^{k}_{t}};t<\zeta\right)
+\bfE^{\rm res}_x\left(e^{-A^{k}_{\zeta}};t\ge\zeta\right)\right)\\ 
&\longrightarrow\ \|h^+\|_\infty\left(\bfE^{\rm res}_x\left(e^{-A^{k}_{\infty}};\zeta=\infty\right)
+\bfE^{\rm res}_x\left(e^{-A^{k}_{\zeta}};\zeta<\infty\right)\right),\ \ t\to\infty.
\end{align*}
The equation in parentheses equals 0 by Corollary \ref{masa} and the assumption in this proposition.
\end{proof}

Put
\bequ\la{solu}
\cS={\cS}^{\rm sub}\cap{\cS}^{\rm sup}.
\eequ
Applying Proposition \ref{p-l}, we have a next Liouville property:

\bt\la{Li}
If $\bfP_x(\zeta=\zeta^i<\infty)=1$, then every function
 in ${\cS}$ is zero {\rm q.e.} $x$.
\et

\bl\la{g-ex}
The function $\bfE^{\rm res}_x(e^{-A^k_\zeta};\zeta<\infty)$, $\bfE^{\rm res}_x(e^{-A^k_\infty};\zeta=\infty)$
 and $\bfE^{\rm res}_x(e^{-A^k_\zeta})$ belong to $\cF_{\loc}^\dagger$.
\el
\begin{proof}
Put $g(x)=\bfE^{\rm res}_x(e^{-A^k_\zeta};\zeta<\infty)$. By the Markov property
\begin{align*}
p_tg(x)&=\bfE^{\rm res}_x\left(e^{-A^{k}_{t}-A^{k}_{\zeta(\theta_t)}(\theta_t)}
;t<\zeta, \zeta(\theta_t)<\infty\right)\\
&=\bfE^{\rm res}_x\left(e^{-A^{k}_{\zeta}};t<\zeta<\infty\right)\le g,
\end{align*}
that is, $g$ is $p_t$-excessive. Let $\varphi$ be a strictly positive function in $L^2(E;m)\cap C(E)$ and 
 define $g_n=g\wedge nR_1\varphi$. Then $g_n\in L^2(E;m)$, 
 $g_n\le nR_1\varphi(\in\cF)$ and $e^{-t}p_tg_n\le g_n$, $t\ge 0$, and thus  
 $g_n\in\cF$ by \cite[Lemma 2.3.2]{FOT}. Since $O_n:=\{R_1\varphi>1/n\}$ are open sets
 satisfying $O_n\uparrow E$ as 
  $n\to\infty$ and $g=g_n$ on $O_n$, $g$ belongs to $\cF_{\loc}$. On account of the boundedness 
  of $g$, $g\in\cF_{\loc}^\dagger$.
  
  By the same argument above, we can show that the other functions belong to $\cF_{\loc}^\dagger$.
\end{proof}

\bl\la{g-ex1}
The functions $\bfE^{\rm res}_x(e^{-A^k_\zeta};\zeta<\infty)$, 
$\bfE^{\rm res}_x(e^{-A^k_\infty};\zeta=\infty)$ and $\bfE^{\rm res}_x(e^{-A^k_\zeta})$ belong to ${\cS}$.
\el
\begin{proof}
Put $g(x)=\bfE^{\rm res}_x(e^{-A^k_\zeta};\zeta<\infty)$. Let $\varphi\in\cF_+\cap C_0(E)$ and 
 $\{\psi_n\}\subset\cF_+\cap C_0(E)$ the sequence defined in Lemma \ref{G-tildeH}. Since
\begin{align*}
p_t(g\psi_n)(x)&=\bfE^{\rm res}_x\left(e^{-A^k_t}\bfE^{\rm res}_{X_t}
\left(e^{-A^k_{\zeta}}1_{\{\zeta<\infty\}}\right)\psi_n(X_t)\right)\\
&=\bfE^{\rm res}_x\left(e^{-{A^k_t-A^k_{\zeta(\theta_t)}(\theta_t)}}1_{\{t<\zeta,\zeta(\theta_t)<\infty\}}
\psi_n(X_t)\right)\\
&=\bfE^{\rm res}_x\left(e^{-{A^k_\zeta}}1_{\{\zeta<\infty\}}\psi_n(X_t)\right),
\end{align*}
\begin{align*}
(g\psi_n-p_t(g\psi_n),\varphi)_m&=(g-p_t(g\psi_n),\varphi)_m\\
&=\left(\bfE^{\rm res}_x\left(e^{-A^k_\zeta}1_{\{\zeta<\infty\}}(1-\psi_n(X_t))\right),\varphi\right)_m.
\end{align*}
Hence
\begin{align*}
0\le(g\psi_n-p_t(g\psi_n),\varphi)_m&\le\left(\bfE^{\rm res}_x\left(1-\psi_n(X_t)\right),\varphi\right)_m\\
&\le(1-p^{\rm res}_t\psi_n,\varphi)_m\\
&=(\psi_n-p^{\rm res}_t\psi_n,\varphi)_m,
\end{align*}
and thus $0\le\cE(g\psi_n,\varphi)\le\cE^{\rm res}(\psi_n,\varphi)$. Noting that $1\in\cF_{\loc}^\dagger$, 
we see from Lemma \ref{G-tildeH} that 
$$
\lim_{n\to\infty}\cE^{\rm res}(\psi_n,\varphi)=\cE^{\rm res}(1,\varphi)=0
$$
and so $\lim_{n\to\infty}\cE(g\psi_n,\varphi)=0$. Using Lemma \ref{G-tildeH} again, we have 
$\cE(g,\varphi)=0.$

 By the same argument above, we can show that the other functions also belong to ${\cS}$.
\end{proof}

Theorem \ref{Li} and Lemma \ref{g-ex1} lead to 

\bc\la{e-alpha}
The next statements are equivalent{\rm :}

\medskip
{\rm \ (i)} $\bfP_x(\zeta=\zeta^i<\infty)=1$.

{\rm (ii)} If $h\in{\cS}$, then $h=0$, {\rm q.e.}
\ec

\begin{proof}
The implication from (i) to (ii) follows from Theorem \ref{Li}. On account of Lemma \ref{g-ex1}, It follow from (ii) that $\bfP_x(\zeta=\zeta^i<\infty)=1$ q.e. 
Noting that $g(x):=\bfE_x(e^{-A^k_\zeta})$ is an excessive function, $p_tg(x)\le g(x)$ and $p_tg(x)\uparrow g(x)$,  we see that $g$ is finely continuous and conclude that $g(x)=0$ for 
all $x\in E$ and so $\bfP_x(\zeta=\zeta^i<\infty)=1-g(x)=1$ for all $x\in E$.
\end{proof}

\br\la{ess-re1}
For $\lambda>0$, denote by $X^{(\lambda)}=(\bfP_x^{(\lambda)},X_t.\zeta)$ be the $\lambda$-killing process of $X$,
that is the Hunt process generated by $(\cE_\lambda,\cF)$. By Remark \ref{ess-re}, 
  (SCI) of $X$ is equivalent to (SCI$_\lambda$), $\bfP^{(\lambda)}_x(\zeta=\zeta^i<\infty)=1$. 
 Hence we see from Corollary \ref{e-alpha} that (SCI) of $X$ is equivalent to that if 
 $h\in{\cS}_\lambda(:={\cS}_\lambda^{\rm sub}\cap{\cS}_{\lambda}^{\rm sup})$, then $h=0$, {\rm q.e.}, where ${\cS}_\lambda^{\rm sub}$ and 
 ${\cS}_{\lambda}^{\rm sup}$ are the space of subsolutions and supersolutions associated with $(\cE_\lambda,\cF)$.
\er

\be
For a non-negative function $k$ in $L^1_{\loc}(\bR^d)$ define 
$$
\cE(u,v)=\frac{1}{2}\int_{\bR^d}(\nabla u,\nabla v)dx+\int_{\bR^d}u^2k(x)dx, \ \ u\in \bH^1(\bR^d)\cap C_0(\bR^d).
$$
Here $\bH^1(\bR^d)$ is 
the Sobolev space of order 1.  Let $\cF$ be the closure of the form above and 
denote by ${\bf P}_x$ the process generated by $(\cE,\cF)$. 
Then $\bfP^{\rm res}_x$ is
the Brownian motion $({\bf P}^B_x,X_t)$ on $\bR^d$ and  ${\bf P}_x$ is the killed process of ${\bf P}^B_x$ 
by the multiplicative functional $\exp({-\int_0^tk(X_s)ds})$. Since the Brownian motion is conservative,
${\bf P}_x$ satisfies (SCI), and ${\bf P}_x$ satisfies the (EK) if and only if 
$$
 {\bf P}^B_x\left(\int_0^\infty k(X_t)dt=\infty\right)=1,\ \ x\in\bR^d.
$$
Let $h$ be a bounded, twice continuously 
differentiable function on $\bR^d$ satisfying $-(1/2)\Delta h(x)+k(x) h(x)=0$, in particular, 
 $-(1/2)\Delta h(x)+\lambda h(x)=0\ (\lambda>0)$, then  $h(x)\equiv 0$ by Corollary \ref{e-alpha}. 

If 
$$
 {\bf P}^B_x\left(\int_0^\infty k(X_t)dt<\infty\right)>0,\ \ x\in\bR^d,
$$
then $P_x(\zeta=\zeta^i<\infty)>0$ and $P_x(\zeta=\zeta^i<\infty)+P_x(\zeta=\infty)=1$.

\ee

 \be \la{bog}
Let $\bar{E}$ be a compactification of $E$. Suppose that a subsolution $h$ is continuous on
 $\bar{E}$ and
 $\sup_{x\in E}h(x)(=\max_{x\in \bar{E}}h(x))\ge 0$. If $\bfP^{\rm res}_x(\tau_E<\infty)=1$, 
 the argument similar to that of Lemma \ref{k-shl} and Proposition \ref{ex-res} leads that 
 $$
 h(x)\le\bfE_x\left(h(X_{\tau_E-});\tau_E=\tau_E^p<\infty\right),
$$ 
where $\tau_E$ is the first exit time from $E$; $\tau_E=\inf\{t>0\mid X_t\not\in E\}$.
Since $h(X_{\tau_E-})\in\partial E:=\bar{E}\setminus E$ on $\{\tau_E=\tau_E^p<\infty\}$, 
the maximum of $h$ on $\bar{E}$ 
attains at a point in 
$\partial E$:
$$
h(x)\le\sup_{x\in\partial E}h(x),\ \ x\in E.
$$

Consider the absorbing symmetric 
$\alpha$-stable process $X^D=(\bfP^D_x,X_t,\tau_D)$ on a bounded Lipschitz open set $D\subset\bR^d$ 
(cf. \cite{BBC}). Write for $(\cE^D,\cF^D)$ 
the Dirichlet form generated by $X^D$. Then its killing measure $k^D$ is
$$
k^D(dx)=k^D(x)dx,\ \ k^D(x)=C\left(\int_{\bR^d\setminus D}\frac{1}{|x-y|^{d+\alpha}}dy\right),
$$
where $C$ is a constant depending on $d$ and $\alpha$. It is shown in \cite[Theorem 1.1]{BBC} that 
$\bfP^{D,{\rm res}}_x(\tau_D<\infty)=1$ if and only if $\alpha>1$. Note that  in \cite{BBC} they
call the resurrected process $\bfP^{D,{\rm res}}_x$ a {\it cencored stable process} and prove that 
the two processes are identical (\cite[Theorem 2.1]{BBC}). Let 
$h$ be a function in $\cS^{\rm sub}$ associated with $(\cE^D,\cF^D)$. 
We see that for $\alpha\le 1$ the resurrected process is recurrent because of 
its conservativeness and the finiteness of the Lebesgue measure of $D$. More strongly, 
it is Harris recurrent by  Absolute Continuity Condition (AC) (\cite[Lemma 4.8.1]{FOT}). 
Hence, it follow from \cite[Chapter X, Proposition (3.11)]{RY} that for $\alpha\le 1$
 $$
 \bfP^{D,{\rm res}}_x\left(\int_0^\infty k^D(X_t)dt=\infty\right)=1,
 $$
and so $h(x)\le 0$.   
 If $\alpha>1$ and $h$ continuous on the closure $\bar{D}$ of $D$, then
 $$
 h(x)\le\bfE^{D,{\rm res}}_x\left(e^{-A^k_{\tau_D}}h(X_{\tau_D-})\right).
 $$ 

\ee
 	
 \section{strong Maximum principle} 
  Let $h$ be a nontrivial subsolution in $\cS_-^{\rm{sub}}$ and put $G=\{x\in E\mid h(x)<0\}$. Then 
  $G$ is a quasi-open set and so its capacity is positive, Cap$(G)>0$. Thus by the irreducibility ({I}) and 
  the absolute continuity  ({AC}) of $X$,
  \bequ\la{irrd}
  \bfP_x(\sigma_G<\zeta)>0\ \  \text{for all}\ x\in E,
  \eequ	
  where $\sigma_G=\inf\{t>0\mid X_t\in G\}$ (\cite[Theorem 4.7.1]{FOT}). As a result, 
  $$ 
  \bfP_x\left(\int_0^\zeta(h1_G)(X_t)dt<0\right)>0\ \ \text{for all}\ x\in E,
  $$
  and thus by Theorem \ref{eq} 
  \begin{align*}
  h(x)&\le\int_0^\infty e^{-t} h(x) dt\le \int_0^\infty e^{-t}p_t h(x) dt\le  \int_0^\infty e^{-t}p_t (h1_G)(x) dt\\
  &=\bfE_x\left(\int_0^\zeta (h1_G)(X_t)dt\right)<0\ \  \text{for all}\ x\in E.
  \end{align*}

   \begin{thm}\la{BP-1}
   	Suppose that $X$ satisfies $({{I}})$ and $({{SF}})$.
   	
   	\smallskip
   	\noindent
   	{\rm (i)} Let $h\in\cS^{\rm sub}\cap C(E)$ and $M=\sup_{x\in E}h(x)\ge 0$.
   	Then $h\equiv M$ or $h(x)<M$ for all $x\in E$.\\   	
   	{\rm (ii)} Let $h\in\cS^{\rm sup}\cap C(E)$ and $m=\inf_{x\in E}h(x)\le 0$.
   	Then $h\equiv m$ or $h(x)>m$ for all $x\in E$.
       \end{thm}
\begin{proof}
We will only prove (i). Since $\cE(h-M,\varphi)=\cE(h,\varphi)-M\cE(1,\varphi)\le 0$ for any $\varphi\in
\cF_+\cap C_0(E)$, $h-M$ is in 
${\cS}_-^{\rm{sub}}$
and so in $\widetilde{\cS}_-^{\rm{sub}}$ by Theorem \ref{eq}.
  Therefore, we see from the argument above that $h(x)-M\le p_t(h-M)(x)<0$, and 
  thus $h(x)<M\ \text{for all}\ x\in E$.
\end{proof}
   
 \br
Theorem \ref{BP-1} (i) says that if $h\in\cS^{\rm sub}$, then
$\sup_{E\setminus K}h(x)=\sup_{E}h(x)$  for any compact set $K$.
\er

\bigskip
 We set 
\begin{align*}
{\bS}&=\left\{\{x_n\}_{n=1}^\infty \subset E\mid
\text{
$\lim_{n\to\infty }\bfE_{x_n}\left(e^{-\zeta}\right)=1$}
\right\},\\
\widetilde{\bS}&=\left\{\{x_n\}_{n=1}^\infty\subset E\mid
\text{
$\lim_{n\to\infty }\bfP_{x_n}(\zeta >\epsilon )\to 0$
 for $\forall \epsilon >0$}
\right\}.
\end{align*}

\bigskip

\begin{lem}{\rm (\cite[Lemma 3.1]{T-B})}\la{S=S} It holds that  
$$
\bS=\widetilde{\bS}.
$$
\end{lem}

\bigskip
We introduce
\begin{equation*}
{\cS}_0^{\rm sub}=\left\{h\in{\cS}^{\rm sub}\cap C(E)\mid \varlimsup_{n\to\infty }h(x_n)\leq 0,\ \ 
\forall\{x_n\}_{n=1}^\infty \in \bS\right\}.
\end{equation*}
Note that a function in ${\cS}_0^{\rm sub}$ is supposed to be continuous.  
Imitating the definition in \cite{BNV}, we define the {\it refined maximum principle} for $\cL$ as follows:

\bigskip
\noindent
\ \ \ \ (${\bf{RMP}}$)\ \ If $h\in {\cS}_0^{\rm sub}$, then $h(x)\leq 0$ for all $x\in E$.

\bigskip
\bl \la{s-i} Suppose that $\bfP_x(\zeta<\infty)=1$. 
For $x\in E$, there exists an increasing sequence $\{\sigma_n\}_{n=1}^\infty$ of stopping times such that
 \bequ\la{sequ}
\{X_{\sigma_n}\}_{n=1}^\infty\in \bS\ \ \bfP_x\text{\rm -a.s.}\ {\rm on}\ \displaystyle{\cap_{n=1}\{\sigma_n<\zeta\}}.
\eequ
\el

\begin{proof}
Let $\{K_l\}_{l=1}^\infty$ be an increasing sequence of compact sets with 
$\cup_{l=1}^\infty K_l=E$ and $\tau_l$ the 
 first exit time $K_l$. By the strong Markov property and $\{\tau_l<\zeta\}\in\sF_{{\tau_l}\wedge
 \zeta}=\sF_{\tau_l}$ (\cite[p.415(f)]{D}),
 \begin{align*}
 \bfE_x\left(\left(1-\bfE_{X_{\tau_l}}
 \left(e^{-\zeta}\right)\right);\tau_l<\zeta\right)&=\bfE_x\left(\bfE_x
 \left(\left(1-e^{-\zeta(\theta_{\tau_l})}\right)1_{\{\tau_l<\zeta\}}
 \Big|\sF_{\tau_l}\right)\right) \\
 &=\bfE_x\left(\left(1-e^{-(\zeta-\tau_l)}\right);\tau_l<\zeta\right).
 \end{align*}
 Since $\tau_l\uparrow\zeta$ as $l\to\infty$, the right hand side above tends to 0, and 
 thus  $\{\bfE_{X_{\tau_l}}\left(e^{-\zeta}\right)\}$ converges to 1 in 
 $L^1(\bfP_x;\Omega^p)$. Hence there exists a subsequence $\{\sigma_n\}$ of $\{\tau_l\}$, 
 $\{X_{\sigma_n}\}_{n=1}^\infty\in \bS\ \ \bfP_x\text{-a.s.}\ \text{on}\  
 \cap_{n=1}\{\sigma_n<\zeta\}$.  
\end{proof}

   \begin{thm}\la{BP}  Suppose that $\bfP_x(\zeta<\infty)=1$. 
   If $X$ satisfies $({{\rm I}})$ and $({{\rm SF}})$, then $({{\rm RMP}})$ holds.
       \end{thm}
  
 \begin{proof}
Let $h\in {\cS}_0^{\rm sub}$ and $\{\sigma_n\}$ a sequence of stopping times defined in 
Lemma \ref{invariant-0}. We can suppose that
$$
\dis{\{X_{\sigma_n}\}_{n=1}^\infty\in\bS\ \ \bfP_x\text{-a.s.}\ \text{on}\ \cap_{n=1}\{\sigma_n<\zeta\}.}
$$
Then by (\ref{s})
\begin{align}\la{main-eq}
h(x)&\le\varlimsup_{t\to\infty}\bfE^{\rm res}_x\left(e^{-A^k_{t\wedge \sigma_n}}h^+(X_{t
 			\wedge \sigma_n})\right)\le\bfE^{\rm res}_x\left(e^{-A^k_{\sigma_n}}h^+(X_{\sigma_n})\right)\\
 			&=\bfE_x\left(h^+(X_{\sigma_n})\right)\nonumber
\end{align}
and so by the definition of ${\cS}_0^{\rm sub}$
$$
h(x)\le\varlimsup_{n\to\infty}\bfE_x\left(h^+(X_{\sigma_n})\right)\le
\bfE_x\left(\varlimsup_{n\to\infty}h^+(X_{\sigma_n})\right)=0.
$$ 			
\end{proof}


\section{Maximum principle for Schr\"odinger forms}

Under {{(AC)}}, there exists a non-negative, jointly measurable $\alpha$-resolvent kernel $r_{\alpha}(x,y)$: 
$$
R_\alpha f(x)=\int_Er_\alpha (x,y)f(y)m(dy),\ x\in E,\ f\in {\mathscr B}_b(E).
$$
Moreover, $r_\alpha (x,y)$ is $\alpha $-excessive in $x$ and in $y$ (\cite[Lemma 4.2.4]{FOT}).
We  simply write $r(x,y)$ for $r_{0}(x,y)$. 
For a measure $\mu$, we define the $\alpha $-potential of $\mu$ by
\begin{equation*}
R_{\alpha} \mu(x) = \int_{E} r_{\alpha}(x,y) \mu(dy).
\end{equation*}
We  write $R\mu$ for $R_{0}\mu$. 

\medskip

We call a Borel measure $\mu$ on $E$ {\it smooth in the strict sense} if there exists a sequence $\{E_n\}$ of Borel sets such that 
for each $n$, $1_{E_n}\cdot\mu\in \cS_{00}$ where $\cS_{00}$ is the set 
of finite, positive Radon measure of finite energy with bounded 1-potential, $\sup_{x\in E}R_1\mu(x)<\infty$  (\cite[Theorem 2.2.4]{FOT}), 
and
$$
\bfP_x(\lim_{n\to\infty}\sigma_n\geq\zeta)=1,\ \ \forall x\in E,
$$
where $\sigma_n=\inf\{t>0\mid X_t\in E\setminus E_n\}$. In particular, a Radon measure 
$\mu$ with $\sup_{x\in E}R_1\mu(x)<\infty$ is smooth in the strict sense.
We denote by $\cS$ the set of smooth measures in the strict sense.

\begin{defi}\label{def-Kato}
Let $\mu\in \cS$.
\begin{enumerate}[(1)]
	\item $\mu$ is said to be in the {\it Kato class} of $X$
	($\cK$ in abbreviation)
	if 
	\begin{equation*}
	\lim_{\alpha \to \infty} \sup_{x\in E}R_\alpha\mu(x) = 0.
	\end{equation*}
	$\mu$ is said to be in the {\it local Kato class} ($\cK_{\loc}$ in abbreviation)
	if for any compact set $K$,  $1_K\cdot \mu$ belongs to $\cK$.	
	\item Suppose that $X$ is transient. A measure $\mu$ is said
	to be in the class $\cK_{\infty}$ if for any $\eps > 0$, 
	there exists a compact set $K = K(\eps)$ 
	\begin{equation*}
	\sup_{x \in E} R(1_{K^c} \mu)(x) < \eps.
	\end{equation*}
	$\mu$ in $\cK_{\infty}$ is called {\it Green-tight}. 
\end{enumerate}
\end{defi}

In this section, we assume that $X$ is transient in addition to $({{I}})$ and $({{SF}})$. For  
a measure in $\mu\in\cK_\infty$ define the Schr\"odinger form $\cE^\mu$ by
$$
\cE^\mu(u)=\cE(u)-\int_E\widetilde{u}^2d\mu,\ \ u\in\cF.
$$
Here $\widetilde{u}$ is a quasi-continuous version of $u$. 
 In the sequel, we always assume that every function $u\in\cF$ is represented by its
quasi-continuous version.
We see that for $\mu\in\cK$ $(\cE^\mu,\cF)$ is a lower bounded closed 
symmetric form (\cite[Theorem 3.1]{Stol-Vo}).
We denote by $\cL^\mu$ the self-adjoint 
operator associated with a Schr\"odinger form $(\cE^\mu,\cF)$,
$$
(-\cL^\mu u,v)_m
$$
It is known that for $\mu\in\cK$ the right hand side is a lower bounded closed symmetric form
and the semigroup $\{p^\mu_t\}$ generated by $\cL^\mu$ is written as a Feynman-Kac semigroup
$$
p^\mu_tf(x)=\bfE_x\left(e^{A^\mu_t}f(X_t)\right), 
$$
where $A^\mu_t$ is the PCAF in the Revuz correspondence to $\mu$. $p^\mu_t$ has 
also the strong Feller property (\cite{Ch}).
A function $h\in \cF^\dagger_{\loc}\cap L^\infty_{\loc}$ is 
called a {\it solution} ({\it subsolution}, {\it supersolution}) to
$\cL^\mu u=0$ if 
$$
\cE^\mu(h,\varphi)=0 \ (\leq 0,\ \geq 0) \ \text{for any}\ \varphi\in\cF_+\cap C_{0}(E).
$$
We write ${\bf S}^\mu\,({\bf S}^{\mu,\rm sub},\ {\bf S}^{\mu,\rm sup})$ for 
the space of {solutions} ({subsolutions}, {supersolutions}) 
and introduce the function spaces:
\begin{align}\la{sol}
\cS^{\mu,\rm sub}&= \{h\in {\bf S}^{\mu,\rm sub}\mid   \|h^+\|_\infty<\infty\},\\
\cS^{\mu,\rm sup}&= \{h\in {\bf S}^{\mu,\rm sup}\mid  \|h^-\|_\infty>-\infty\},\\
\widetilde{\cS}^{\mu,\rm sub}&= \{ h\in\cF^\dagger_{\loc}\cap L^\infty_{\loc}\mid  
 \|h^+\|_\infty<\infty,\   p^\mu_th \ge h\ {\rm q.e.}\},\\
\widetilde{\cS}^{\mu,\rm sup}&= \{ h\in\cF^\dagger_{\loc}\cap L^\infty_{\loc}\mid  
 \|h^-\|_\infty>-\infty,\   p^\mu_th \le h\ {\rm q.e.}\}.
\end{align}

Applying It$\hat{\rm o}$'s formula to $e^{-A^k_{t\wedge \sigma_n}
+A^\mu_{t\wedge \sigma_n}}h(X_{t\wedge \sigma_n})$, 
we have the next inequality similarly to (\ref{s}):
 		\begin{align}\la{s-1}
 			h(x)\le\bfE^{\rm res}_x\left(e^{-A^k_{t\wedge \sigma_n}
 			+A^\mu_{t\wedge \sigma_n}}h(X_{t\wedge \sigma_n})\right)\ \ \text{q.e.}\ x .
 		\end{align} 	
Using the equation (\ref{s-1}), we can extend results  in Section 2 and Section 3 to 	
$\cS^{\mu,\rm sub}$ and $\widetilde{\cS}^{\mu,\rm sub}$ 
($\cS^{\mu,\rm sup}$ and $\widetilde{\cS}^{\mu,\rm sup}$).	In particular, we have the next theorem.
   \begin{thm}\la{eq-1}
 It holds that
 $$
 {\cS}_-^{\mu,\rm{sub}}=\widetilde{\cS}_-^{\mu,\rm{sub}},\ \ \ \
 {\cS}_+^{\mu,\rm{sup}}=\widetilde{\cS}_+^{\mu,\rm{sup}}.
 $$ 
       \end{thm}		

\medskip
 Let $g^\mu(x)$ be the so-called {\it gauge function}: 
\begin{equation}\la{gauge-fun}
g^\mu(x) = \bfE_x \left(e^{A_{\zeta}^{\mu}}\right).
\end{equation}
Define 
\begin{equation}\la{j-ground3}
	\lambda (\mu)= \inf \left\{ \cE(u)\, \Big|\  u \in
	\cF,\ \int_E u^2 d\mu = 1\right\}.
\end{equation}
We then know in \cite[Theorem 5.1]{C} that for $\mu\in\cK_\infty$
\bequ\la{gau-th}
\lambda (\mu)>1\ \Longleftrightarrow\ \sup_{x\in E}g^\mu(x)<\infty.
\eequ

\begin{lem}\la{tsu-0} {\rm (\cite[Lemma 5.2]{T3})} If $\lambda (\mu)>1$, then the function
 $g^\mu$ is $p^{\mu}_{t}$-excessive, i.e.,  
	$p^{\mu}_tg^\mu(x)\uparrow g^\mu(x)$ as $t\downarrow 0$.
\end{lem}

\begin{lem}{\rm (\cite[Lemma 5.4]{T3})}\label{lem-2-2}
 If $\lambda (\mu)>1$, then $g^\mu$ belongs to ${\cF}^\dagger_{\loc}\cap C(E)$.
\end{lem}

\medskip
We introduce the function space of strictly positive $p^\mu_t$-excessive functions by
\begin{equation}\la{mu-exc}
\widetilde{\cS}_{++}^{\text{sup}}(\mu) = \{h\in{\cF}^\dagger_{\loc}\cap C(E)\mid h>0,\    p_t^{\mu} h \le h\}.
\end{equation}
We see that for $h\in\widetilde{\cS}_{++}^{\text{sup}}(\mu)$ the bilinear form 
$(\cE^{\mu,h},\cF^{\mu,h})$ on $L^2(E;h^2m)$ defined by
\begin{equation}\la{h-form}
\left\{
\begin{split}
& \cE^{\mu,h} (u) = \cE^{\mu} (hu)  \\
& \cF^{\mu,h} = \{ u \in L^2(E;h^2m)\mid hu\in \cF\}
\end{split}
\right.
\end{equation}
is a {regular Dirichlet form} on $L^2(E;h^2m)$. This fact can be proved in the manner of
\cite[Theorem 6.3.2]{FOT}. As a result, if $\cS_{++}^{\text{sup}}(\mu)$ 
is not empty, then $(\cE^\mu,\cF)$ is positive semi-definite,
$$
\cE^{\mu} (u) = \cE^{\mu,h} (u/h)\geq 0, 
$$
consequently, $\gamma(\mu)$ in (\ref{j-ground3}) is non-negative.
 Lemma \ref{tsu-0} and Lemma \ref{lem-2-2} tell us that   
if $\lambda(\mu)>1$, then the gauge function $g^\mu$ belongs to 
$\widetilde{\cS}_{++}^{\text{sup}}(\mu)$. Hence, we have
\bl For $\mu\in\cK_\infty$
\bequ\la{r-g-g}
\lambda(\mu)>1\ \Longrightarrow\ \gamma(\mu)\ge0.
\eequ
\el

\medskip
Let $h\in\cS^{\mu,\rm sub}$. Since for $\varphi\in\cF^{\mu,g^\mu}_+\cap C_0(E)$, 
$g^\mu\varphi$ belongs to $\in\cF_+\cap C_0(E)$ by Lemma \ref{lem-2-2},  
the function $h/g^\mu$ satisfies
$$
\cE^{\mu,g^\mu}(h/g^\mu,\varphi)=\cE^{\mu}(h,g^\mu\varphi)\le 0,\ 
\ \forall \varphi\in\cF^{\mu,g^\mu}_+\cap C_0(E),
$$
that is, $h/g^\mu$ is a subsolution of $\cE^{\mu,g^\mu}$.
Therefore, noting that $1\le g^\mu\le M<\infty$, we see from Theorem \ref{BP-1} (i) that
\begin{align}\la{e-q}
\sup_{x\in E}h(x)=0&\ \Longleftrightarrow\ \sup_{x\in E}\frac{h}{g^\mu}(x)=0\nonumber\\
&\ \Longleftrightarrow\ 
\frac{h}{g^\mu}(x)\equiv 0\ \text{or}\ \frac{h}{g^\mu}(x)<0\ \text{for all}\ x\in E\\
&\ \Longleftrightarrow\ 
h(x)\equiv 0\ \text{or}\ h(x)<0\ \text{for all}\ x\in E.\nonumber
\end{align}

\bigskip
\noindent
The equation (\ref{e-q}) can be shown without the condition $\lambda(\mu)>1$. Indeed, we have 
   \begin{thm}\la{BP-2}
   Suppose that $X$ satisfies $({{I}})$ and $({{SF}})$.\\   	
   	{\rm (i)} Let $h\in\cS_-^{\mu,\rm sub}\cap C(E)$ such that $\sup_{x\in E}h(x)=0$.
   	Then $h\equiv 0$ or $h(x)<0$ for all $x\in E$.\\   	
   	{\rm (ii)} Let $h\in\cS_+^{\mu,\rm sup}\cap C(E)$ such that $\inf_{x\in E}h(x)=0$.
   	Then $h\equiv 0$ or $h(x)>0$ for all $x\in E$.
 \end{thm}
   	 \begin{proof} 
We will give only the proof of (i).  By the same argument as that stated in the previous paragraph of Theorem 
\ref{BP-1} we can prove it.  
Suppose that there exists $x_0\in E$ such that $h(x_0)<0$. Then $G=\{x\in E\mid h(x)<0\}$
 is an open set. Thus by the irreducibility ({I}) and 
  the absolute continuity  ({AC}) of $X$,
  \bequ\la{irrd}
  \bfP_x(\sigma_G<\zeta)>0, \ \  x\in E,
  \eequ	
  where $\sigma_G=\inf\{t>0\mid X_t\in G\}$ (\cite[Theorem 4.7.1]{FOT}). As a result, 
  $$ 
  \bfP_x\left(\int_0^\zeta1_G(X_t)dt>0\right)>0,\ \  x\in E,
  $$
  Let
$$
\cO(\omega)=\{t\in[0,\infty)\,\mid\, h(X_t(\omega))<0\}.
$$
Then 
\begin{align*}
\bfP_x\left(\int_0^\infty e^{-t}e^{A^\mu_t}h(X_t)dt\le\int_{\cO(\omega)}e^{-t}e^{A^\mu_t}h(X_t)dt<0\right)>0,
\end{align*}
and thus
$$
\bfE_x\left(\int_0^\infty e^{-t}e^{A^\mu_t}h(X_t)dt\right)<0,\ \ x\in E.
$$
Noting that $p^\mu_th(x)\ge h(x)$ for $h\in\cS^{\mu,\rm sub}_-$ by the argument similar to 
that in Lemma \ref{invariant}, we have 
\begin{align*}
h(x)&=\int_0^\infty e^{-t}h(x)dt\le\int_0^\infty e^{-t}p^\mu_th(x)dt\\
&=\bfE_x\left(\int_0^\infty e^{-t}e^{A^\mu_t}h(X_t)dt\right)<0,\ \  x\in E.
\end{align*}
\end{proof}

\begin{lem}{\rm (\cite[Lemma 3.16]{T4})}\la{e-1}
	For $\mu\in \cK_{\infty}$ with $\lambda(\mu)=1$, there exists 
	a positive function $h\in \cF_{e}\cap C_b(E)$ ($\subset {\cF}^\dagger_{\loc}\cap C(E)$) such that 
	$$
	\cE^\mu(h,\varphi)=0\  \ \text{for all}\ \varphi\in\cF.
	$$
	Moreover, $h$ is $p^\mu_t$-invariant, $h=p^\mu_th$.
\end{lem}

\begin{lem}\la{e-2}
	For $\mu\in \cK_{\infty}$ with $\lambda(\mu)\le 1$, there exists a positive function 
	$h\in \cF_{e}\cap C_b(E)$ such that $h\in{\cS}^{\mu,\rm sub}$ and $h\le p_t^\mu h$.
	\end{lem}
\begin{proof}
Noting that by the definition of $\lambda(\mu)$
$$
	\lambda (\lambda(\mu)\mu)= \inf \left\{ \cE(u)\, \Big|\  u \in
	\cF,\ \lambda(\mu)\int_E u^2 d\mu = 1\right\}=1,
$$
we see from Lemma \ref{e-1} above that there exists a positive function $h\in \cF_{e}\cap C_b(E)$ such that
$$
\cE^{\lambda(\mu)\mu}(h,\varphi)=0\ \   \ \text{for all}\ 
\varphi\in\cF\cap C_0(E).
$$
Since $\lambda(\mu)\mu\le\mu$ and the function $h$ is $p^{\lambda(\mu)\mu}_t$-invariant,  we have 
$$
\cE^{\mu}(h,\varphi)\le\cE^{\lambda(\mu)\mu}(h,\varphi)=0\ \   \ \text{for all}\ 
\varphi\in\cF_+\cap C_0(E)
$$
and
$$
h=p^{\lambda(\mu)\mu}_th\le p^{\mu}_th.
$$
\end{proof}

\medskip

We introduce
\begin{equation*}
{\cS}_0^{\mu,\rm sub}=\left\{h\in{\cS}^{\mu,\rm sub}\cap C(E)\, \Big|\, \varlimsup_{n\to\infty }h(x_n)\leq 0,\ \ 
\forall\{x_n\}_{n=1}^\infty \in \bS\right\}.
\end{equation*}
We define the {refined maximum principle} for $\cL^\mu$ similarly to $\cL$:

\bigskip
\noindent
\ \ \ \ (${\bf{RMP}}$)\ \ If $h\in {\cS}_0^{\mu,\rm sub}$, then $h(x)\leq 0$ for all $x\in E$.

\bigskip

\begin{thm}\la{RMP} Suppose that $\bfP_x(\zeta<\infty)=1$. Then for $\mu\in\cK_\infty$ 
$$
\lambda (\mu)>1\ \Longleftrightarrow\  {\rm({RMP})}.
$$
\end{thm}
 \begin{proof}
Suppose $\lambda(\mu)>1$ and $h\in {\cS}_0^{\mu,\rm sub}$.
Let $\{\sigma_n\}$ be a sequence of stopping times such that
$$
\dis{\{X_{\sigma_n}\}\in\bS\ \ \bfP_x\text{-a.s.}}\ \text{on}\ \cap_{n=1}^\infty\{\sigma_n<\zeta\}. 
$$
Then on account of (\ref{gau-th}) we can apply the reverse Fatou lemma to (\ref{s-1})  and have
\begin{align*} 
h(x)&\le\varlimsup_{t\to\infty}\bfE^{\rm res}_x\left(e^{-A^k_{t\wedge \sigma_n}
 			+A^\mu_{t\wedge \sigma_n}}h^+(X_{t\wedge \sigma_n})\right)\\
&\le\bfE^{\rm res}_x\left(e^{-A^k_{\sigma_n}}e^{A^\mu_{\sigma_n}}
h^+(X_{\sigma_n})\right)=\bfE_x\left(e^{A^\mu_{\sigma_n}}h^+(X_{\sigma_n})\right)\ \ {\rm q.e.}
\end{align*}
Applying the reverse Fatou lemma again, we obtain
$$
h(x)\le\varlimsup_{n\to\infty}\bfE_x\left(e^{A^\mu_{\sigma_n}}h^+(X_{\sigma_n})\right)
\le\bfE_x\left(e^{A^\mu_{\zeta}}\varlimsup_{n\to\infty}h^+(X_{\sigma_n})\right)=0\ \ {\rm q.e. }
$$
which implies $h(x)\le 0$ for all $x\in E$.

Suppose $\lambda(\mu)\le 1$ and let $h$ be the function in Lemma \ref{e-2}. 
By H\"older inequality
\begin{align*}
h(x)&\le\bfE_x\left(e^{A^\mu_t}h(X_t)\right)\le\|h^+\|_\infty\left(\sup_{x\in E}
\bfE_x\left(e^{pA^\mu_t}\right)\right)^{1/p}\bfP_x(t<\zeta)^{1/q},
\end{align*}
and $\sup_{x\in E}\bfE_x\left(e^{pA^\mu_t}\right)<\infty$ because of $p\mu\in\cK$. Hence
 for any $\{x_n\}\in\widetilde{\bS}$,   
$\varlimsup_{n\to\infty}h(x_n)\le 0$ and so $h$ belongs 
to ${\cS}_0^{\mu,\rm sub}$; however, since $h(x)>0$ for all $x\in E$, 
{\rm({RMP})} does not hold.
\end{proof}

\br\la{exten}
By the equation (\ref{s-1}) and the gaugeability, $\sup_{x\in E}g^\mu(x)<\infty$, every result in 
Section 2 and Section 3, Accordingly, Lemma \ref{con-res}, Corollary \ref{p-l}, Theorem \ref{Li}  
especially can be extended to 
$\cS^{\mu,\rm sub}$, $\cS^{\mu,\rm sup}$ ($\widetilde{\cS}^{\mu,\rm sub}$, 
$\widetilde{\cS}^{\mu,\rm sup}$) under the condition $\lambda(\mu)>1$.
\er


\bl\la{conti}
For $\mu\in\cK_\infty$ define  
\bequ\la{La-0}
\Lambda(\theta)=\inf\left\{\cE(u)\,\Big|\,\int_Eu^2d(\mu+\theta m)=1\right\},\ \theta\ge 0.
\eequ
If the basic measure $m$ also belongs to $\cK_\infty$, 
then $\Lambda(\theta)$ is continuous.
\el
\begin{proof}
We know from \cite[Theorem 4.8]{T-Com} that for $\nu\in\cK_\infty$ 
the extended Dirichlet space $(\cF_e,\cE)$
 is compactly embedded in
 $L^2(E;\nu)$. Hence, for $\theta_1,\theta_2\ge 0$ there exist functions $u_1, u_2$ in $\cF_e$ such that 
\begin{align*}
\Lambda(\theta_1)&=\cE(u_1),\ \ \ \int_Eu_1^2d(\mu+\theta_1 m)=1,\\
\Lambda(\theta_2)&=\cE(u_2),\ \ \ \int_Eu_2^2d(\mu+\theta_2 m)=1.
\end{align*}
Put 
\begin{align*}
k_1&=\int_Eu_1^2d(\mu+\theta_2 m)=1+(\theta_2-\theta_1)\int_Eu_1^2dm,\\
k_2&=\int_Eu_2^2d(\mu+\theta_1 m)=1+(\theta_1-\theta_2)\int_Eu_2^2dm,
\end{align*}
and $v_1=u_1/\sqrt{k_1}$, $v_2=u_2/\sqrt{k_2}$.
We then have
$$
\Lambda(\theta_1)\le\cE(v_2)=\frac{1}{k_2}\Lambda(\theta_2),\ \ 
\Lambda(\theta_2)\le\cE(v_1)=\frac{1}{k_1}\Lambda(\theta_1),
$$
and thus 
$$
k_1\Lambda(\theta_2)\le\Lambda(\theta_1)\le\frac{1}{k_2}\Lambda(\theta_2).\ \ 
$$
Therefore, noting that $\lim_{\theta_2\to\theta_1}k_1=1$ and $\lim_{\theta_2\to\theta_1}k_2=1$, we have 
$$
\varlimsup_{\theta_2\to\theta_1}\Lambda(\theta_2)\le\Lambda(\theta_1)
\le\varliminf_{\theta_2\to\theta_1}\Lambda(\theta_2),
$$
which implies the continuity of $\Lambda(\theta)$.
\end{proof}

Write $\gamma(\mu)$ for the principal eigenvalue of $\cL^\mu$: 
\begin{equation}\la{j-ground4}
	\gamma (\mu)= \inf \left\{ \cE^\mu(u)\, \Big|\  u \in
	\cF,\ \int_E u^2 dm = 1\right\}.
\end{equation}

\bt\la{r-g}
If $m$ and $\mu$ are in $\cK_\infty$, then
\bequ\la{l-g}
\lambda(\mu)>1\ \Longleftrightarrow\ \gamma(\mu)>0.
\eequ
\et
\begin{proof}
If $\lambda(\mu)>1$, there exists $\theta_0>0$ such that 
$$
\inf\left\{\cE(u)\,\Big|\, \int_Eu^2d(\mu+\theta_0m)=1\right\}=1.
$$
Indeed, 
denote by $u_\theta\in\cF_e$ the function attaining the infimum in (\ref{La-0}),  
and put $k=\int_Eu_0^2dm,\ v_\theta=u_0/\sqrt{1+k\theta}$. Note that
$$
\int_Ev_{\theta}^2d(\mu+\theta m)=\frac{1}{1+k\theta}\left(\int_Eu_0^2d\mu+\theta\int_E u_0^2dm\right)=1.
$$
Then
since
$$
\Lambda(\theta)=\cE(u_\theta)\le\cE(v_\theta)=\frac{1}{1+k\theta}\cE(u_0)
=\frac{1}{1+k\theta}\Lambda(0),
$$
we have $\lim_{\theta\to\infty}\Lambda(\theta)=0$. If $\Lambda(0)=\lambda(\mu)>1$, then 
there exists $\theta_0$ such that $\Lambda(\theta_0)=1$ by Lemma \ref{conti}.

For $\varphi\in\cF\cap C_0(E)$ define 
$$
G(t)=\frac{\cE(u_{\theta_0}+t\varphi)}{\int_E(u_{\theta_0}+t\varphi)^2d(\mu+\theta_0m)}.
$$
Then $G'(0)=0$ and thus
$$
\cE(u_{\theta_0},\varphi)-\int_Eu_{\theta_0}\varphi d(\mu+\theta_0m)=0\ \Longleftrightarrow\ 
\cE^\mu(u_{\theta_0},\varphi)=\theta_0\int_Eu_{\theta_0}\varphi dm.
$$
Hence $\theta_0$ equals $\gamma(\mu)$, the principal eigenvalue 
of the Schr\"odinger operator $-\cL^\mu$, and $v_0=u_{\theta_0}/\|u_{\theta_0}\|_m$
 is the normalized principal eigenfunction. Hence
$$
\gamma(\mu)=\cE(v_0)-\int_Ev_0^2d\mu\ge(\lambda(\mu)-1)\int_Ev_0^2d\mu>0.
$$

Suppose $\gamma(\mu)>0$ and let $u_0$ be the function attaining $\lambda(\mu)$. Then
$$
\cE(u_0)-\int_Eu_0^2d\mu-\gamma(\mu)\int_Eu_0^2dm\ge 0,
$$
and thus
$$
\lambda(\mu)=\frac{\cE(u_0)}{\int_Eu_0^2d\mu}\ge 1+\gamma(\mu)
\frac{\int_Eu_0^2dm}{\int_Eu_0^2d\mu}>1.
$$
\end{proof}

\br
We suppose that the basic measure $m$ is finite and, in addition, that the operator norm $\|p_t\|_{1,\infty}$ of 
 the semi-group $p_t$ 
 is bounded by $C(t)(<\infty)$ such that $\int_\delta^\infty C(t)dt<\infty$ for any $\delta>0$. 
Here $\|\cdot\|_{1,\infty}$ is the operator norm from $L^1(E;m)$ to $L^\infty(E;m)$.
Then for a compact set $K$
\begin{align*}
R1_{E\setminus K}(x)&=\int_0^\delta p_t1_{E\setminus K}(x)dt+
\int_\delta^\infty p_t1_{E\setminus K}(x)dt\\
&=\delta+\left(\int_\delta^\infty C(t)dt\right)m(E\setminus K).
\end{align*}
Hence, for any $\epsilon$ there exists $\delta>0$ and a compact set $K$ such that 
the right hand side above is less than $\epsilon$, which implies $m\in\cK_\infty$.
In particular, for a uniformly elliptic self-adjoint operator $\cL$ of form
$$
\cL=\frac{\partial}{\partial x_i}\left(a_{ij}(x)\frac{\partial}{\partial x_j}\right)
$$
in a bounded domain $D\subset \bfR^d$, the Lebesgue measure is in $\cK_\infty$. In this case, 
for $\mu\in\cK_\infty$, $\gamma(\mu)>0$ is equivalent with $\lambda(\mu)>1$. 

In general, $\gamma(\mu)>0$ implies $\lambda(\mu)>1$; however, the converse does not always 
hold (cf.
\cite[Remark 3.1]{T-B}).
\er

\smallskip
We set 
$$
{\bS}'=\left\{\{x_n\}_{n=1}^\infty \subset E\mid
\text{For any compact set $K$, $\exists N$ s.t. $x_n\not\in K$ $\forall n>N$}
\right\}
$$
and put
\bequ\la{ba-s}
\partial^Mh=\inf_{K\in\cK}\sup_{E\setminus K}h(x),
\eequ
where $\cK$ is the totality of compact sets of $E$.
   \begin{thm}\la{BP1}
Suppose that $X$ satisfies $({\rm{I}})$, $({\rm{SF}})$ and let $\mu\in\cK_\infty$ with $\lambda(\mu)>1$. 
 Then for $h\in\cS^{\mu,\rm sub}\cap C(E)$
$$
\partial^Mh\le 0\ \Longleftrightarrow\ \sup_{x\in E}h(x)\le 0.
$$
In particular, if $\varlimsup_{n\to\infty}h(x_n)\leq 0\ 
\text{for any}\ \{x_n\}\in\bS$, then $\varlimsup_{n\to\infty }h(x_n)\leq 0\ \text{for any}\ \{x_n\}\in\bS'$.
   
       \end{thm}
 \begin{proof} It holds that 
 \begin{align*}
\partial^Mh\le 0&\ \Longleftrightarrow\ \varlimsup_{n\to\infty}h(x_n)\leq 0\ 
\text{for any}\ \{x_n\}\in\bS'\\
 &\ \Longrightarrow \ \varlimsup_{n\to\infty }h(x_n)\leq 0\ \text{for any}\ \{x_n\}\in\bS \\
 &\ \Longrightarrow\ \sup_{x\in E}h(x)\le 0\ \Longrightarrow \partial^Mh\le 0.
\end{align*} 
\end{proof}


\br
Let  $\mu=\mu^+-\mu^-$ be a signed measure. We treat $\cE^{\mu^+}=\cE+\int_Eu^2d\mu^+$ and $\mu^-$
instead of $\cE$ and $\mu$ in the argument above, in other words, consider $k+\mu^+$ the killing measure.
Suppose that $(\cE^{\mu^+},\cF\cap C_0(E))$ is closable and denote by $(\cE^{\mu^+},\cF^{\mu^+})$  
 the closure.  
If the Hunt process $X^{\mu^+}$ generated by 
the regular Dirichlet form $(\cE^{\mu^+},\cF^{\mu^+})$ satisfies ({\rm I}) and ({\rm SF}) and $\mu^-\in
\cK_\infty$ satisfies 
\begin{equation}\la{j-ground4}
	\lambda (\mu)= \inf \left\{ \cE^{\mu^+}(u)\, \Big|\  u \in
	\cF\cap C_0(E),\ \int_E u^2 d\mu^- = 1\right\}>1,
\end{equation}
then the results above can be extended in this case. 

\er

 \end{document}